\documentclass[final,onefignum,onetabnum]{siamart190516}

\usepackage{braket,amsfonts}

\usepackage{array}
\usepackage{multirow}
\usepackage[caption=false]{subfig}
\usepackage{pgfplots}

\newsiamthm{claim}{Claim}
\newsiamremark{remark}{Remark}
\newsiamremark{hypothesis}{Hypothesis}
\crefname{hypothesis}{Hypothesis}{Hypotheses}

\usepackage{algorithmic}

\usepackage{graphicx,epstopdf}
\usepackage[caption=false]{subfig} 

\Crefname{ALC@unique}{Line}{Lines}

\usepackage{amsopn}

\usepackage[caption=false]{subfig} 
\usepackage{float}
\usepackage{xspace}
\usepackage{bold-extra}
\usepackage[most]{tcolorbox}

\colorlet{texcscolor}{blue!50!black}
\colorlet{texemcolor}{red!70!black}
\colorlet{texpreamble}{red!70!black}
\colorlet{codebackground}{black!25!white!25}

\newcommand{\Real}{\mathbb{R}}
\usepackage{cite}

\graphicspath{{figures/}}

\lstdefinestyle{siamlatex}{%
  style=tcblatex,
  texcsstyle=*\color{texcscolor},
  texcsstyle=[2]\color{texemcolor},
  keywordstyle=[2]\color{texemcolor},
  moretexcs={cref,Cref,maketitle,mathcal,text,headers,email,url},
}

\tcbset{%
  colframe=black!75!white!75,
  coltitle=white,
  colback=codebackground, 
  colbacklower=white, 
  fonttitle=\bfseries,
  arc=0pt,outer arc=0pt,
  top=1pt,bottom=1pt,left=1mm,right=1mm,middle=1mm,boxsep=1mm,
  leftrule=0.3mm,rightrule=0.3mm,toprule=0.3mm,bottomrule=0.3mm,
  listing options={style=siamlatex}
}

\newtcblisting[use counter=example]{example}[2][]{%
  title={Example~\thetcbcounter: #2},#1}

\newtcbinputlisting[use counter=example]{\examplefile}[3][]{%
  title={Example~\thetcbcounter: #2},listing file={#3},#1}

\DeclareTotalTCBox{\code}{ v O{} }
{ 
  fontupper=\ttfamily\color{black},
  nobeforeafter,
  tcbox raise base,
  colback=codebackground,colframe=white,
  top=0pt,bottom=0pt,left=0mm,right=0mm,
  leftrule=0pt,rightrule=0pt,toprule=0mm,bottomrule=0mm,
  boxsep=0.5mm,
  #2}{#1}

\patchcmd\newpage{\vfil}{}{}{}
\title{Improved stochastic rounding \thanks{This research was funded by the EU ECSEL Joint Undertaking under grant agreement no.~826452.}}
\author{Lu~Xia\footnotemark[2] , Martijn Anthonissen\footnotemark[2] , Michiel Hochstenbach\footnotemark[2] , Barry~Koren\thanks{Department of Mathematics and Computer Science, Eindhoven University of Technology, PO Box 513, 5600 MB Eindhoven, The Netherlands (\email{l.xia1@tue.nl}, \email{m.j.h.anthonissen@tue.nl}, \email{m.e.hochstenbach@tue.nl}, \email{b.koren@tue.nl}).} }

\ifpdf
\hypersetup{ pdftitle={Comparison of different rounding modes} }
\fi

\begin{document}
\maketitle


\begin{abstract}
	Due to the limited number of bits in floating-point or fixed-point arithmetic, rounding is a necessary step in many computations. Although rounding methods can be tailored for different applications, round-off errors are generally unavoidable. When a sequence of computations is implemented, round-off errors may be magnified or accumulated. The magnification of round-off errors may cause serious failures. Stochastic rounding (SR) was introduced as an unbiased rounding method, which is widely employed in, for instance, the training of neural networks (NNs), showing a promising training result even in low-precision computations. Although the employment of SR in training NNs is consistently increasing, the error analysis of SR is still to be improved. Additionally, the unbiased rounding results of SR are always accompanied by large variances. In this study, some general properties of SR are stated and proven. Furthermore, an upper bound of rounding variance is introduced and validated. Two new probability distributions of SR are proposed to study the trade-off between variance and bias, by solving a multiple objective optimization problem. In the simulation study, the rounding variance, bias, and relative errors of SR are studied for different operations, such as summation, square root calculation through Newton iteration and inner product computation, with specific rounding precision. 
\end{abstract}




\section{Introduction}
\label{sec:intro}

In many computations, rounding is an unavoidable step, due to the limited number of bits in floating-point or fixed-point arithmetic. Many rounding schemes have been proposed and studied for different applications, such as floor, ceiling, round to the nearest, stochastic rounding, etc. These rounding modes normally have different round-off errors. When a sequence of computations is implemented, round-off errors may be accumulated and magnified. In real world problems, the magnification of round-off errors may cause severe failures. In pursuit of high accuracy, high-precision computations are generally employed, for which computing times may be long. 

To reduce computing times, low-precision computing is becoming increasingly popular, especially in the area of machine learning. In \cite{higham2019squeezing}, algorithms are proposed to squeeze matrices from double or single precision to half precision, using two-sided diagonal scaling.
In \cite{higham2019simulating}, some low-precision simulation results are compared for different rounding methods, e.g., directed rounding, rounding to nearest, and stochastic rounding. The detailed numerical analysis of each rounding method still needs to be developed.

An unbiased stochastic rounding (SR) scheme was applied in \cite{gupta2015deep} to train neural networks (NNs) using low-precision fixed-point arithmetic. The experiments show that where the deterministic rounding scheme fails, the training results using 16-bit fixed-point representation with the SR method are very similar to those computed in 32-bit floating-point precision. Inspired by \cite{gupta2015deep}, SR is  widely employed in training NNs in low-precision floating-point or fixed-point precision, see, e.g., \cite{na2017chip,ortiz2018low,wang2018training}. Although the employment of SR in training NNs is increasing, the error analysis of SR is still to be completed. Additionally, the unbiased rounding results of SR always have large variances. 

In this paper, SR is studied with respect to two aspects. First, numbers are rounded to a specific number of fractional bits and an upper bound of rounding variance is introduced and validated. Next, some general properties of SR are also proven. To study the trade-off between variance and bias of rounding results using SR, two new probability distributions are proposed, for which a multi-objective optimization problem (MOP) is formulated. The probability can be easily optimized according to user requirements on variance and bias, for instance by particle swarm optimization (PSO). Since the new probability distributions minimize both variance and bias, these rounding modes are potentially interesting for training NNs and numerical solution algorithms.

The remainder of the paper is organized as follows. The rounding rules of some general deterministic rounding methods are summarized in \cref{sec:DRmodels}. \Cref{sec:SRmodels} outlines the scheme of stochastic rounding and introduces the formula of the variance and some general properties. Then, in \cref{sec:probdistribution}, new probability distributions are proposed to discuss the trade-off between variance and bias, by solving the MOP. Furthermore, numerical simulation results of bias, variance and absolute value of relative errors are presented in \cref{sec:numericalstudy}, using different rounding schemes for summation, square root calculation through Newton iteration and inner product computation. Finally, conclusions are drawn in \cref{sec:conclusion}.  
\section{Deterministic rounding}\label{sec:DRmodels}
In this section, the schemes of some general deterministic rounding methods, such as directed rounding to an integer and rounding to the nearest integer, are summarized. The directed rounding is normally used in interval arithmetic and comprised of four rounding methods, e.g., rounding down (floor) and rounding up (ceiling), etc. The floor method rounds a number $x$ to the largest integer smaller than $x$, vice versa for ceiling. The rounding-to-the-nearest methods vary in different tie breaking rules, such as round half up, round half down, round half to even, round half to odd, etc.~\cite{kahan1996ieee}. Round half up is commonly used in financial calculations \cite{cowlishaw2003decimal}, where numbers smaller than half are rounded down and those larger than or equal to half are rounded up, vice versa for round half down. Rounding half to even is also called convergent rounding (CR), which is the default rounding mode used in IEEE 754 floating-point operations. It eliminates bias by rounding different numbers towards or away from zero. In contrast, rounding half to odd is rarely employed in computations, since it will never round to zero \cite{santoro1989rounding}. A summary of the aforementioned rounding schemes is given in \cref{tab:roundingmode}, together with some examples. 
\begin{table}[h!]
	{\footnotesize
		\caption{Summary of different deterministic rounding methods, and four illustrative examples.}\label{tab:roundingmode}
		\begin{center}
			\begin{tabular}{l|p{6cm}|l|l|l|l}
				\cline{1-6} \rule{0pt}{2.3ex}%
				\centering{Rounding mode}          & Rounding rule                        & 1.6&  0.5 &$-0.5$&$-1.6$ \\ \cline{1-6}	\rule{0pt}{2.3ex}%
				round down               & round toward negative infinity          & 1  & 0   & $-1$ &$-2$                     \\
				round up                 & round toward positive infinity          & 2  & 1   & $\phantom{-}0$  &$-1$\\ 
				round half up            & round to the nearest integer with tie rounding toward positive infinity   & 2& 1   & $\phantom{-}0$ & $-2$                                   \\
				round half down          & round to the nearest integer with tie rounding toward negative infinity   & 2& 0   & $-1$ &$-2$   \\
				round half to even       & round to the nearest integer with tie rounding toward the nearest even number & 2& 0   & $\phantom{-}0$  &$-2$\\
				round half to odd        & round to the nearest integer with tie rounding toward the nearest odd number  & 2& 1   & $-1$ &$-2$\\\cline{1-6}
			\end{tabular}
		\end{center}
	}
\end{table}
\section{Stochastic rounding}
\label{sec:SRmodels}
In this section, some properties of stochastic rounding (SR) are stated and proven. Furthermore, an algorithm is introduced to round numbers to a specific fractional digit, and the variance bound of the algorithm is proposed and validated.

The SR method is studied in \cite{nightingale1986gap,allton1989stochastic,price1993stochastic} and has recently been applied in \cite{gupta2015deep}. It is widely employed in training NNs \cite{hopkins2019stochastic,wang2018training}. Compared to deterministic rounding methods, it provides an unbiased rounding result by setting a probability that is proportional to the proximity of $x$. The definition of stochastic rounding is the following:
\begin{definition}
	\label{def:stochasticrounding}
	Let $x \in \Real$, and let $\delta$ be the rounding precision. Then the rounded value $\mathrm{fl}(x)$ of $x$ using SR is defined as
	\begin{align}
	\mathrm{fl}(x)= \begin{cases}
	\lfloor x \rfloor,  \quad  \quad  \quad  ~~~~    &\text{with probability}~p_1(x)=1-\frac{x-\lfloor x \rfloor}{\delta},\\
	\lfloor x \rfloor+\delta,  \quad  \quad  \quad  &\text{with probability}~p_2(x)=\frac{x-\lfloor x \rfloor}{\delta},
	\end{cases}
	\label{eq:stroundinggeneral}
	\end{align}
	where $\lfloor x \rfloor$ indicates the greatest representable floating/fixed-point number less than or equal to $x$ \cite{gupta2015deep}. 
\end{definition} 
For instance, if a floating-point number $0.4$ is rounded to integer with SR, so the rounding precision $\delta=1$, then 0.4 will be rounded down to 0 with probability 0.6 and rounded up to 1 with probability 0.4.
\subsection{General properties} 
In this subsection, some properties of SR will be proven. The first result is straightforward and is briefly introduced in \cite{allton1989stochastic}.

Let $\mathrm{fl}(x)$ be the random variable corresponding to the rounding process. Let $p_i$ be the corresponding probability of $\mathrm{fl}(x)=x_i$. 
\begin{corollary} 
	\label{co: ExpectedValueofSRrounding}
	The expected value of the rounded value $\mathrm{fl}(x)$ in SR is $x$. This means that the expected rounding error of $\mathrm{fl}(x)$ is 0.
\end{corollary}
\begin{proof}
	The expected value of $\mathrm{fl}(x)$ with discrete probability distribution can be calculated as
	\begin{align*}
	\mathbb{E}(\text{fl}(x)) &= x_1  p_1(x)+ x_2  p_2(x)\\
	&= \lfloor x \rfloor  \left(1-\frac{x-\lfloor x \rfloor}{\delta}\right) + (\lfloor x \rfloor+\delta)  \frac{x-\lfloor x \rfloor}{\delta} \\
	&= \lfloor x \rfloor + x-\lfloor x \rfloor = x,
	\end{align*}
	where $x_1=\lfloor x \rfloor$ and $x_2=\lfloor x \rfloor+\delta$.
\end{proof}
The variance of $\text{fl}(x)$ in rounding scheme \cref{eq:stroundinggeneral} is given by
	\begin{equation}
	\mathit{V}(\mathrm{fl}(x))=( \lfloor x \rfloor- x)^2  p_1(x)+( \lfloor x \rfloor +\delta - x)^2  p_2(x).
	\label{eq:varianceofSRrounding}
	\end{equation}

\subsection{Rounding to a specific number of fractional bits or decimal digits}
When a specific number of fractional bits is required, the rounding result can be easily achieved by multiplying with a scalar $\theta$. For instance, one fractional bit indicates a rounding precision $\delta=2^{-1}$ and the corresponding scalar is $\theta=2$. The procedure for rounding to a specific number of fractional bits is given in \cref{alg:roundspdpfloat}. 
\begin{algorithm}[h!]
	\caption{Round to a specific number of fractional bits or decimal digits.}
	\label{alg:roundspdpfloat}
	\begin{algorithmic}[1]		
		\STATE Definitions: Number of fractional bits: $n$, then scalar $\theta=2^n$ or $10^n$. 
		\STATE The scaled value $\widetilde{x}= \theta x  $.
		\STATE The approximated value of $\widetilde{x}$ is given as \begin{equation}
		\mathrm{fl}(\widetilde{x})= \begin{cases}
		\lfloor \widetilde{x} \rfloor \quad  \quad  \quad  \quad  ~ ~ &\text{with probability}~1-\frac{\widetilde{x}-\lfloor \widetilde{x} \rfloor}{1},\\
		\lfloor \widetilde{x} \rfloor+ 1  \quad  \quad    &\text{with probability}~\frac{\widetilde{x}-\lfloor \widetilde{x} \rfloor}{1},
		\end{cases}
		\label{eq:rounding2spdecimal}
		\end{equation}
		where $\lfloor \widetilde{x} \rfloor$ indicates the largest integer less than or equal to $\widetilde{x}$.
		\STATE Scaling it back, we have $\mathrm{fl}(x)=\frac{\mathrm{fl}(\widetilde{x})}{\theta}$.
	\end{algorithmic}
\end{algorithm}
To the authors' knowledge, the following propositions are not proven in literature. 
\begin{proposition}
	The expected value of rounding results, under the condition of rounding to the specific number of fractional bits $n\in\mathbb{N}$, through stochastic rounding \cref{alg:roundspdpfloat}, is still unbiased. 
\end{proposition}
\begin{proof} If the number of fractional bits is $n$, a scalar can be defined as $\theta=2^{n}$. A random variable $x \in \Real$ can be scaled as $\widetilde{x}= \theta x$ and rounded to \cref{eq:rounding2spdecimal} with different probability distributions.
	
	According to \cref{co: ExpectedValueofSRrounding}, the expected value of rounding results can be calculated by
	\begin{equation*}
	\mathbb{E}(\mathrm{fl}(x))=\mathbb{E}\left(\frac{\mathrm{fl}(\widetilde{x})}{\theta}\right)=\frac{1}{\theta}\mathbb{E}(\mathrm{fl}(\widetilde{x}))=\frac{1}{\theta} \widetilde{x} =\frac{1}{\theta}   \theta  x =x.
	\end{equation*}
	So the bias of $\mathrm{fl}(x)$ is zero.
\end{proof}

\begin{proposition} The variance of rounding to the specific number of fractional bits $n$, using stochastic rounding \cref{alg:roundspdpfloat}, is bounded by $(\frac{1}{2\theta})^2$. 
	\label{prop:variancebound}
\end{proposition}
\begin{proof} According to \cref{eq:varianceofSRrounding}, the variance of $\mathrm{fl}(\widetilde{x})$ obtained by \cref{alg:roundspdpfloat} is
	\begin{align}
	V(\mathrm{fl}(\widetilde{x}))&=(\lfloor \widetilde{x} \rfloor-\widetilde{x})^2 \left(1-\frac{\widetilde{x}-\lfloor \widetilde{x} \rfloor}{1}\right)+(\lfloor \widetilde{x}\rfloor +1-\widetilde{x})^2 \left(\frac{\widetilde{x}-\lfloor \widetilde{x} \rfloor}{1}\right)  \nonumber\\
	&= (\lfloor \widetilde{x} \rfloor-\widetilde{x})^2+(\lfloor \widetilde{x} \rfloor-\widetilde{x})^3+(\widetilde{x}-\lfloor \widetilde{x} \rfloor)-2(\lfloor \widetilde{x} \rfloor-\widetilde{x})^2-(\lfloor \widetilde{x} \rfloor-\widetilde{x})^3  \nonumber\\
	&= (\widetilde{x}-\lfloor \widetilde{x} \rfloor)-(\lfloor \widetilde{x} \rfloor-\widetilde{x})^2  \nonumber\\
	&= -\left(\widetilde{x}-\lfloor \widetilde{x} \rfloor-\frac{1}{2}\right)^2+\frac{1}{4}.
	\label{eq:variancerounding2spcdecimal}
	\end{align}
	Let $\Delta \widetilde{x} =\widetilde{x}-\lfloor \widetilde{x} \rfloor$, then $\Delta \widetilde{x} \in [0,1]$. As a result, \cref{eq:variancerounding2spcdecimal} has the maximum value $\frac{1}{4}$, when $\Delta \widetilde{x} =\frac{1}{2}$. Consequently, $V(\mathrm{fl}(x))=V\big(\frac{1}{\theta}\mathrm{fl}(\widetilde{x})\big)=\frac{1}{\theta^2} V(\mathrm{fl}(\widetilde{x}))\leq\frac{1}{4\theta^2}=\left(\frac{1}{2\theta}\right)^2$.
\end{proof}
To validate \cref{prop:variancebound}, a set of numbers from 0 to 2, with the smallest interval between two consecutive numbers equal to $10^{-4}$, has been rounded 10,000 times, under rounding-to-4-fractional-bit scenario, i.e., $\theta=2^{4}$. The corresponding variances are shown in \cref{fig:variance2decimal}. The blue line indicates the variance calculated using \cref{eq:varianceofSRrounding}, which is almost invisible due to the coverage of the red dashed line, which latter shows the population variance calculated over 10,000 observations using
\begin{subequations}
	\label{eq:variancegen}
	\begin{equation}
	\widehat{V}=\frac{1}{N}\sum_{i=1}^{N}( x_i-\mu )^2,
	\end{equation}
	where $\mu$ is the mean value of a vector $\mathbf{x}$ comprised of $N$ random variables:
	\begin{equation}
	\mu=\frac{1}{N}\sum_{i=1}^{N}x_i.
	\end{equation}
	\end{subequations}
 From the zoomed in subplot around $x=\frac{1}{2^5}=0.03125$ in \cref{fig:variance2decimal}, it can be observed that the blue line and the red dashed line are bounded by $2^{-10}\approx 9.77 \cdot 10^{-4}$, satisfying \cref{prop:variancebound}. According to \cref{eq:variancerounding2spcdecimal}, the variance is zero when $\widetilde{x}=\lfloor \widetilde{x} \rfloor$, with $\widetilde{x}=2^{-4}j$ in \cref{fig:variance2decimal}, where $j\in \mathbb{N}$. 
\begin{figure}[h!]
	\centering
	\includegraphics[width=0.6\columnwidth]{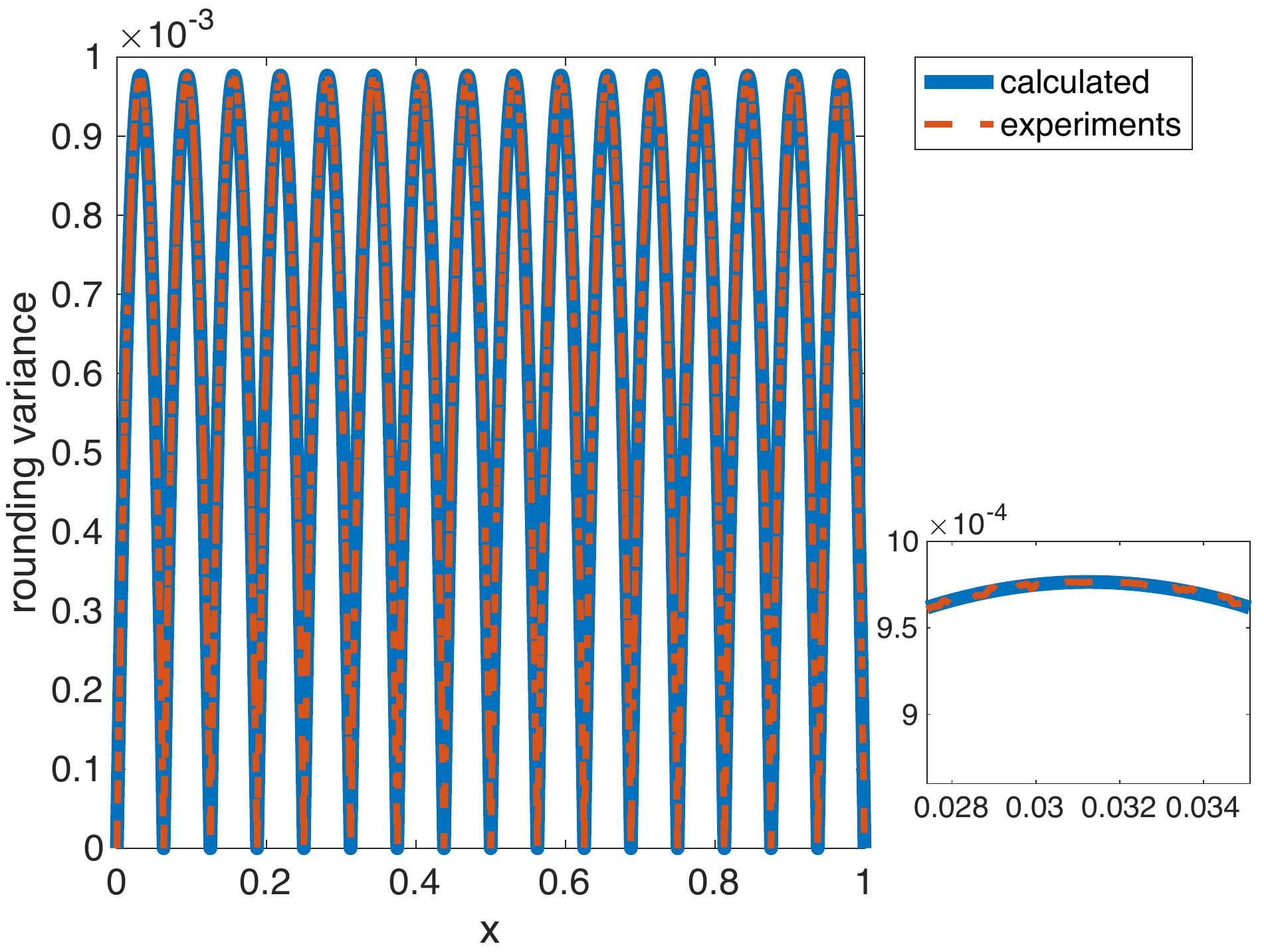}
	\caption{Variance of $x$ from 0 to 2 when $n=4$ (rounding to 4 fractional bits).} 
	\label{fig:variance2decimal}
\end{figure}
\begin{proposition} In stochastic rounding, it holds 
	\begin{equation*}
	\mathrm{fl}(\mathrm{fl}(\cdots \mathrm{fl}(\mathrm{fl}(x_1)~\mathrm{op}~x_2)~\mathrm{op}~\cdots)~\mathrm{op}~x_{N_s})=\mathrm{fl}(x_1)~\mathrm{op}~\mathrm{fl}(x_2)~\mathrm{op}~\cdots~\mathrm{op}~\mathrm{fl}(x_{N_s}),
	\end{equation*}
where $N_s$ indicates the number of terms for $\mathrm{op}\in\{+,-\}$.
	\label{prop: roundsum}
\end{proposition}
\begin{proof}
	Assume $x_1$ and $x_2$ are rounded using the following rounding scheme
	\begin{equation*}
	        \text{fl}(x_i)=\begin{cases}
	        \lfloor x_i \rfloor, \quad  &\text{with probability~} p_i, \\
	        \lfloor x_i \rfloor +\delta, \quad  &\text{with probability~}1-p_i,
	        \end{cases}
	\end{equation*} where $i\in\{1,2\}$.
	Then
	\begin{align*}
	\text{fl}(x_1)+\text{fl}(x_2)=\begin{cases}
	\lfloor x_1 \rfloor +\lfloor x_2 \rfloor, \quad           &\text{with probability~} p_1p_2, \\
	\lfloor x_1 \rfloor +\lfloor x_2 \rfloor +\delta, \quad   &\text{with probability~} (1-p_1)p_2 +p_1(1-p_2),\\
	\lfloor x_1 \rfloor +\lfloor x_2 \rfloor +2\delta, \quad  &\text{with probability~} (1-p_1)(1-p_2).
	\end{cases}
	\end{align*}	
We also have
\begin{align*}
\text{fl}\big(\text{fl}(x_1)+ x_2\big)&=\begin{cases}
\lfloor \lfloor x_1 \rfloor+ x_2 \rfloor =  \lfloor x_1 \rfloor+\lfloor x_2 \rfloor, \\ \quad\text{with probability}~p_1p_2,\\
\lfloor \lfloor x_1 \rfloor+ \delta+ x_2\rfloor = \lfloor x_1 \rfloor+\lfloor x_2 \rfloor +\delta, \\\quad\text{with probability}~(1-p_1)p_2 +p_1(1-p_2),\\
\lfloor \lfloor x_1 \rfloor+ \delta+ x_2\rfloor+\delta=  \lfloor x_1 \rfloor+\lfloor x_2 \rfloor +2\delta, \\ \quad\text{with probability}~(1-p_1)(1-p_2),
\end{cases} \\                                                                                                                           
&=\text{fl}(x_1)~+~ \text{fl}(x_2).
\end{align*}	
For summation of $N_s$ terms,
\begin{align*}
\mathrm{fl}(\mathrm{fl}(x_1)+x_2)+\cdots)+x_{N_s})&=\begin{cases}
\lfloor x_1 \rfloor + \lfloor x_2 \rfloor +\cdots +\lfloor x_{N_s} \rfloor, \\\quad\text{with probability}~p_1p_2\cdots p_{N_s},\\
\lfloor x_1 \rfloor + \lfloor x_2 \rfloor + \cdots +\lfloor x_{N_s} \rfloor +\delta, \\
\quad\text{with probability}~(1-p_1)p_2\cdots p_{N_s}\\+p_1(1-p_2)\cdots p_{N_s}+\cdots+p_1p_2\cdots (1-p_{N_s}),\\
\quad \quad \quad \quad \vdots \\
\lfloor x_1 \rfloor + \lfloor x_2 \rfloor + \cdots +\lfloor x_{N_s} \rfloor +n\delta,\\
\quad\text{with probability}~(1-p_1)(1-p_2)\cdots (1-p_{N_s}),\\
\end{cases}\\
&=\text{fl}(x_1)+ \text{fl}(x_2)+\cdots+\text{fl}(x_{N_s}).
\end{align*}
The above relation also holds for subtraction, by replacing $+$ by $-$.
\end{proof}

The following propositions hold under the rounding to integer scenario, where $\delta=1$.
\begin{proposition}For multiplication using stochastic rounding, it holds
\begin{align}
\mathrm{fl}\big(\mathrm{fl}(x_1)\mathrm{fl}(x_2)\big)=\mathrm{fl}(x_1)\mathrm{fl}(x_2).
\label{eq:fl_multiplication}
\end{align}

	\label{prop:multiplication}
\end{proposition}
\begin{proof}
		\begin{align*}
	\text{fl}(x_1)\text{fl}(x_2)=\begin{cases}
	\lfloor x_1 \rfloor \lfloor x_2 \rfloor, \quad & \text{with probability~}p_1p_2, \\
	\lfloor x_1 \rfloor \lfloor x_2 \rfloor +\lfloor x_2 \rfloor, \quad  &\text{with probability~}(1-p_1)p_2, \\
	\lfloor x_1 \rfloor \lfloor x_2 \rfloor +\lfloor x_1 \rfloor, \quad  &\text{with probability~}p_1(1-p_2),\\
	\lfloor x_1 \rfloor \lfloor x_2 \rfloor +\lfloor x_1 \rfloor +\lfloor x_2 \rfloor +1, \quad  &\text{with probability~}(1-p_1)(1-p_2).
	\end{cases}
	\end{align*}
	Since $\lfloor x_1 \rfloor \lfloor x_2 \rfloor$ is an integer, $\lfloor \lfloor x_1 \rfloor \lfloor x_2 \rfloor \rfloor=\lfloor x_1 \rfloor \lfloor x_2 \rfloor$. Consequently, $\mathrm{fl}\big(\mathrm{fl}(x_1)\mathrm{fl}(x_2)\big)=\mathrm{fl}(x_1)\mathrm{fl}(x_2)$.
	\end{proof}

\begin{proposition} When $x_1, x_2 \in (0,1)$ and $x_1x_2\leq\frac{1}{2}$, the worst-case relative round-off error is larger than or equal to 1 in \cref{eq:fl_multiplication}.
	\label{prop:worsterrorbound_multi}
\end{proposition}

\begin{proof} Applying SR to $x_2$, when $\text{fl}(x_2)=0$, the relative error of \cref{eq:fl_multiplication} is always 1. The worst-case scenario only occurs when $\text{fl}(x_2)=1$. Hence, the result of \cref{eq:fl_multiplication}, in the worst-case scenario, is
	\begin{align*}
		\text{fl}(x_1)\text{fl}(x_2)= \begin{cases}
			\lfloor x_1 \rfloor,\quad &\text{with probability}~ p,\\
			\lfloor x_1 \rfloor+1,\quad &\text{with probability}~ 1-p.
		\end{cases}
	\end{align*}
	The worst-case absolute relative round-off error is 
	\begin{align}
			\label{eq:relaerror}
		e_{\text{worst}}=\Big\vert\frac{ x_1 x_2-\text{fl}(x_1)\text{fl}(x_2)}{x_1x_2}\Big\vert = \begin{cases}
			\big\vert\frac{x_1x_2-\lfloor x_1 \rfloor}{x_1x_2}\big\vert,\quad &\text{with probability}~ p,\\[1mm] 
			\big\vert\frac{x_1x_2-(\lfloor x_1 \rfloor+1)}{x_1x_2}\big\vert,\quad &\text{with probability}~ 1-p.
		\end{cases}
	\end{align}
When $x_1 \in (0,1)$, we have $\lfloor x_1 \rfloor=0$ and \cref{eq:relaerror} becomes
\begin{equation}
\label{eq:eq:relaerror01}
e_{\text{worst}}=\begin{cases}
\big\vert1-\frac{\lfloor x_1 \rfloor}{x_1x_2}\big\vert=1,\quad &\text{with probability}~ p,\\[1mm] 
\big\vert1-\frac{\lfloor x_1 \rfloor}{x_1x_2}-\frac{1}{x_1x_2}\big\vert=\big\vert1-\frac{1}{x_1x_2}\big\vert, &\text{with probability}~ 1-p.
\end{cases}
\end{equation} 
 When $x_1, x_2 \in (0,1)$ and $x_1x_2\leq\frac{1}{2}$, we have $\frac{1}{x_1x_2}\geq2$, so $e_{\text{worst}}=\big\vert1-\frac{1}{x_1x_2}\big\vert\geq1$.
\end{proof}
For $x_1>1$, we have the following.
\begin{proposition}\label{remk:remark1} 
	When $x_1 \in (i,i+1)$, $x_2 \in (0,1)$ and $x_1x_2\leq\frac{i}{2}$, where $i\in \mathbb{N}^{+}$, we find in a similar way as above that the worst-case relative round-off error is larger than or equal to 1 in \cref{eq:fl_multiplication}.
\end{proposition}
\begin{proof}
	For $x_1\in(i,i+1)$, we have
	\begin{equation*}
	e_\text{worst}=\begin{cases}
	\big\vert1-\frac{\lfloor x_1 \rfloor}{x_1x_2}\big\vert=\big\vert1-\frac{i}{x_1x_2}\big\vert,\quad &\text{with probability}~ p,\\[1mm] 
	\big\vert1-\frac{\lfloor x_1 \rfloor+1}{x_1x_2}\big\vert=\big\vert1-\frac{i+1}{x_1x_2}\big\vert,\quad &\text{with probability}~ 1-p.
	\end{cases}.
	\end{equation*}
	 When $x_1x_2\leq\frac{i}{2}$, we have $\big\vert1-\frac{i}{x_1x_2}\big\vert\geq1$ and $\big\vert1-\frac{i+1}{x_1x_2}\big\vert\geq\big\vert -1-\frac{1}{x_1x_2}\big\vert>1$.
\end{proof}
\begin{figure}[tbhp]
	\label{fig:contourplot}
	\centering
	\subfloat[Error with probability $p$]{\label{fig:contourplot1}\includegraphics[width=0.5\textwidth]{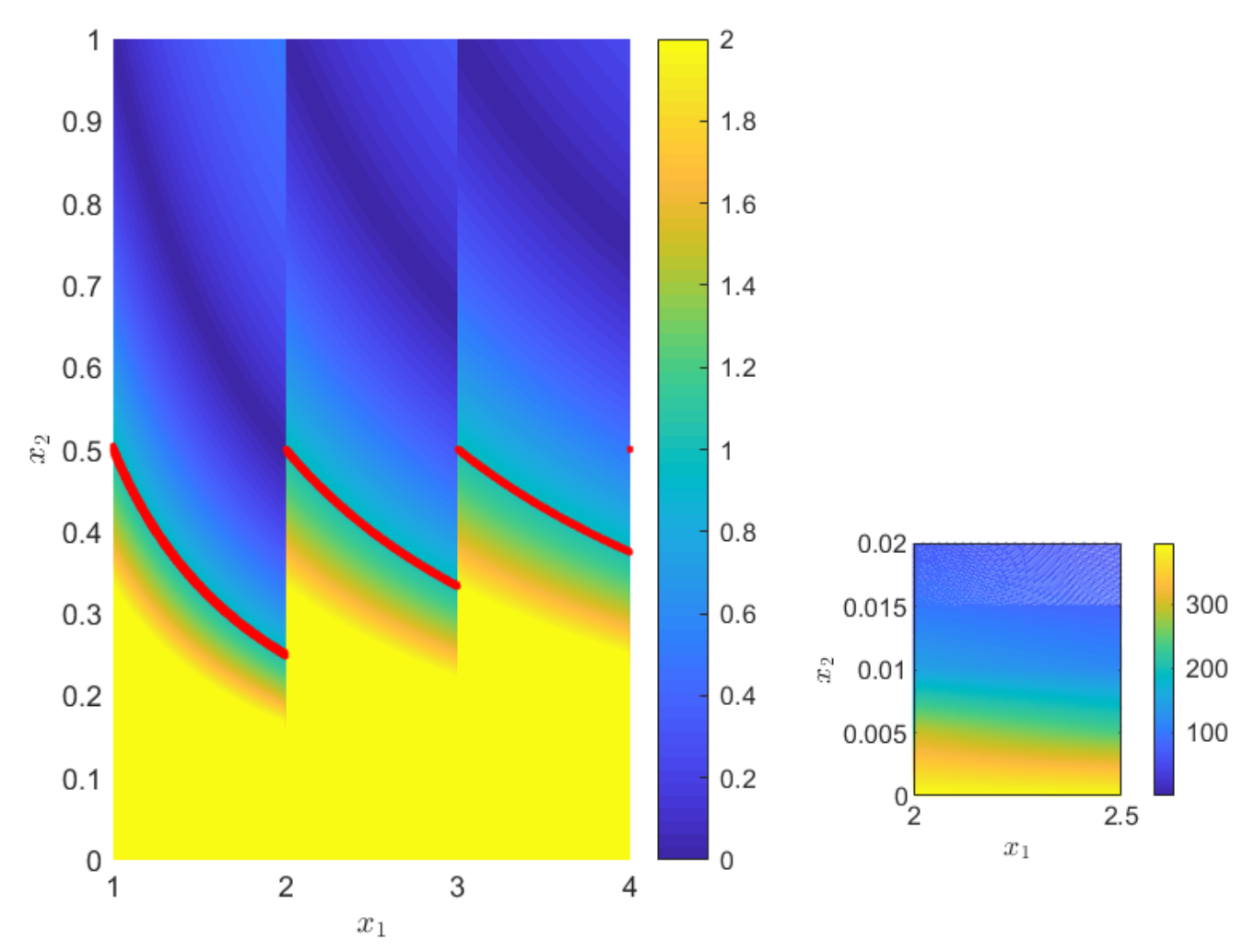}}
	\subfloat[Error with probability $1-p$]{\label{fig:contourplot2}\includegraphics[width=0.5\textwidth]{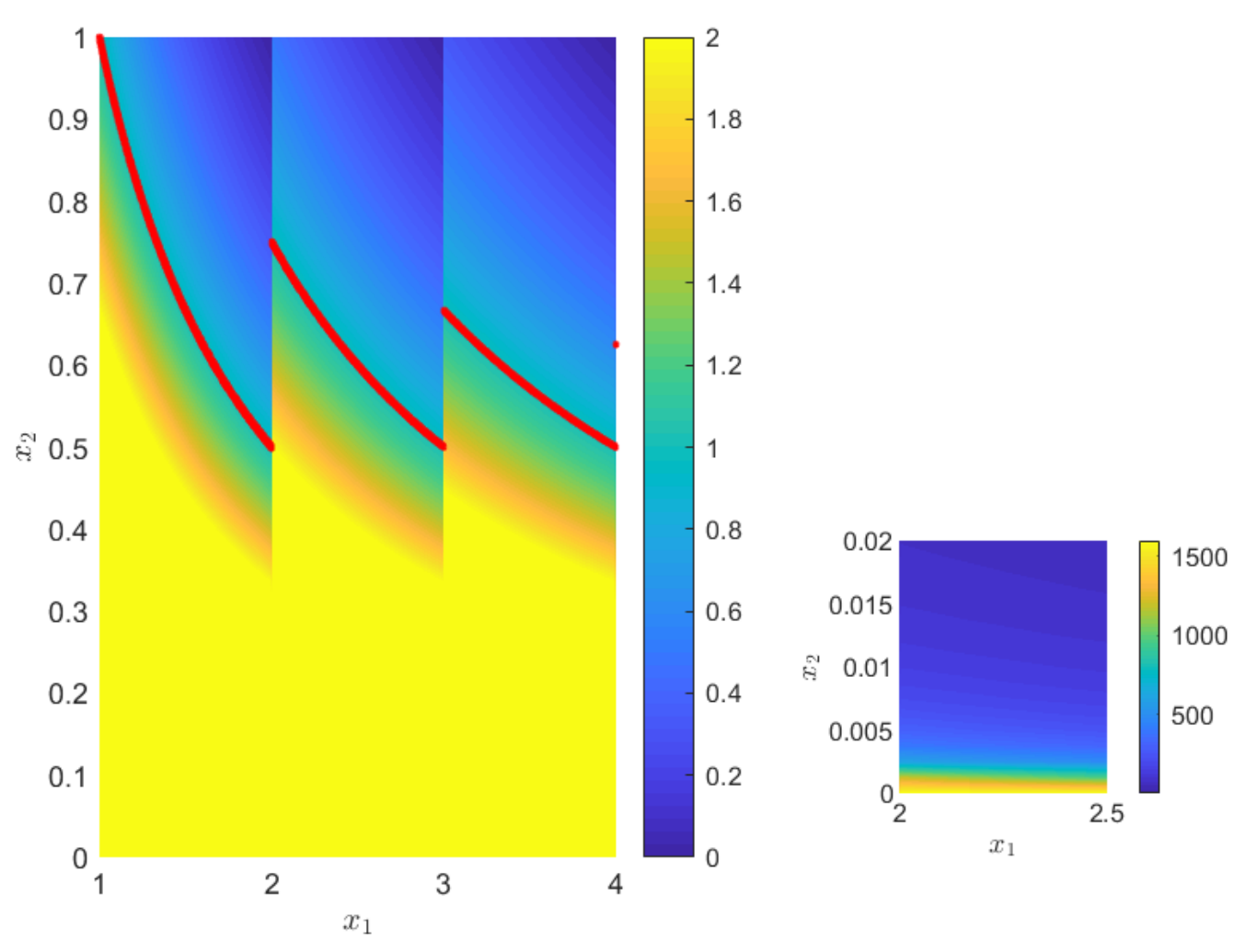}}
	\caption{Contour plots of worst-case relative error with respect to $x_1$ and $x_2$, with probabilities $p$ (a) and $1-p$ (b), where the red lines are the lines with $x_1x_2=\frac{i}{2}$ and $x_1x_2=\frac{i+1}{2}$ in (a) and (b), respectively.}
\end{figure}
\cref{remk:remark1} will be used further in \cref{sec:innerproduct}.
\cref{fig:contourplot} shows the contour plots of the worst-case relative error with respect to $x_1$ and $x_2$, with probabilities $p$ (\cref{fig:contourplot1}) and $1-p$ (\cref{fig:contourplot2}). In the yellow areas, $e_{\text{worst}}\geq2$. At the red lines, $e_{\text{worst}}=1$. Specifically, the red lines are the lines with $x_1x_2=\frac{i}{2}$ and $x_1x_2=\frac{i+1}{2}$ in \cref{fig:contourplot1,fig:contourplot2}, respectively. It can be observed that $e_{\text{worst}}>1$ for both \cref{fig:contourplot1,fig:contourplot2}, when $x_1x_2<\frac{i}{2}$. Furthermore, the worst-case relative error increases when $x_2$ decreases. From the zoomed in subplots (small figures) in \cref{fig:contourplot}, it can be seen that the worst-case relative error can be very large, when $x_2$ is close to 0. For the same value of $x_2$, \cref{fig:contourplot2} shows the larger worst-case relative error than \cref{fig:contourplot1}. The aforementioned worst-case scenario will not occur, if floor rounding and rounding-to-the-nearest methods are employed, because then the numbers will always round towards zero if they are close to zero.
\section{Optimization of the probability distribution of stochastic rounding}
\label{sec:probdistribution}
In this section, a new probability distribution is proposed. To find it, a multi-objective optimization problem (MOP) is formulated. The probability is computed with different emphasis on variance and bias. To do so we use the scalarization method \cite[Ch.~2, p.~11--36]{caramia2008multi}. To meet the requirements of different computations, constraints on bias and variance are realized by adding a penalty function to the objective function, which will be defined in the following section. The optimization problem is solved using PSO.
\subsection{Problem formulation}
Instead of the probabilities in \cref{eq:stroundinggeneral}, a general probability distribution $p$ is considered. The new stochastic rounding with unknown probability, is defined as 
\begin{align}
\mathrm{fl}(x)= \begin{cases}
\lfloor x \rfloor, \quad  \quad  \quad  \quad  &\text{with probability}~p_1=p,\\
\lfloor x \rfloor+\delta,  \quad  \quad  \quad  &\text{with probability}~p_2=1-p,
\end{cases}
\label{eq:strounding}
\end{align}
where $\lfloor x \rfloor$ indicates the greatest representable floating-point or fixed-point number less than or equal to $x$ and where $\delta$ is the rounding precision. 
The variance of rounding scheme \cref{eq:strounding} is 
\begin{equation}
\mathit{V}(p) = (\lfloor x \rfloor-\bar{\mu})^2p+(\lfloor x \rfloor+\delta-\bar{\mu})^2(1-p),
\label{eq:variancestr1}
\end{equation}
where $\bar{\mu}$ is the expected rounding value, given by
\begin{equation}
\bar{\mu}=\lfloor x \rfloor\, p+(\lfloor x \rfloor+\delta)(1-p).
\label{eq:mu}
\end{equation}
Substitute \cref{eq:mu} into \cref{eq:variancestr1}, to find
\begin{equation}
\mathit{V}(p)=\delta^2(p-p^2).
\label{eq:variancestr}
\end{equation}
The bias is 
\begin{equation}
\mathit{B}(p) = (\lfloor x \rfloor p+(\lfloor x \rfloor+\delta)(1-p))-x.
\label{eq:bias}
\end{equation}
To find a trade-off between variance and bias, a MOP can be formulated as 
\begin{subequations}
	\begin{align}
	\underset{p}{\text{minimize}}& ~\mathit{V}(p),~\vert \mathit{B}(p)\vert,\label{eq:costfun}\\
	~\text{subject to}& ~\mathit{V}(p)\leq V_{\text{max}},~\vert \mathit{B}(p)\vert \leq B_{\text{max}},~0 \leq p\leq1,~x\in\mathcal{X},\label{eq:constraint}
	\end{align}
\end{subequations}
where $\mathcal{X}$ denotes the domain of the input variables. $V_{\text{max}}$ and $\mathit{B}_{\text{max}}$ can be set according to the users' own preferences, but feasibility should also be considered. Solutions of such MOPs are generally non-unique, since the objective functions are normally conflicting. In this case, Pareto optimality is often achieved in MOPs \cite{censor1977pareto}. An effective approach to find the trade-off between each conflicting objective function is the scalarization method \cite{caramia2008multi}, in which a single scaled fitness function is formulated. Furthermore, the constraints on variance and bias in \cref{eq:constraint} can be realized by adding a penalty ($k_i$) to the objective function \cref{eq:costfun}, as in \cite{xia2019constrained}, given by
\begin{subequations}
	\label{eq:optfun}
			\begin{align}		
			\underset{p}{\text{minimize}}& ~\big(\theta_1(\mathit{V}(p))^2+\theta_2(\mathit{B}(p))^2+\sum_{i=1}^{2}k_i\mathbb{I}_{[0,\infty)}(g_i)\big),\\
			\text{subject to}& ~0 \leq p\leq1,~x\in\mathcal{X},
			\end{align}
\end{subequations}
with $\theta_1+\theta_2=1$, $g_1=\mathit{V}(p)-V_{\text{max}}$ and $g_2=\vert \mathit{B}(p)\vert-B_{\text{max}}$, where $\mathbb{I}$ is an indicator function, having the value 1 when $g_i \in [0,\infty)$, and 0 elsewhere; and where $k_i$ is a constant, indicating the penalty on the $i$th constraint $g_i$. Here, $k_i$ is chosen to be 0 or sufficiently large to realize an unconstrained or constrained condition for $V$ and $B$, respectively. 

Due to the effect of penalties on the objective function, the gradient of the optimization problem is frequently not available. PSO is a gradient-free approach that is used extensively in solving global optimization problems \cite{PSO2015}. It solves problems by searching the best position among a group (swarm) of the candidate solutions (particles). The goal is to find the globally best position by comparing each particle's own best position to its neighbor's best position. Problem \cref{eq:optfun} can be solved using the same PSO algorithm as in \cite{xia2019constrained}. It should be noted that the choice of the optimization method, for instance PSO, is not the crucial part of this study, since the optimal probability distribution can be calculated offline and is not necessarily computed during each rounding process. This paper provides a method to obtain an improved stochastic rounding method with customized rounding variance and bias. 

\subsection{Four probability distributions}
\label{sec:fourdistribution}
In this section, the optimization problem is solved four times, each time with a different emphasis on variance and bias. The resulting variances and biases are compared.
\subsubsection{Bias minimization}
\label{sec:Bias minimization}
If only bias is minimized in \cref{eq:optfun} without any constraint, $\theta_1$ is set to 0. According to \cref{eq:rounding2spdecimal}, $\delta$ can be chosen as 1 for any number of fractional bits, after proper scaling. The probability distribution, and the corresponding variance and bias are shown in \cref{fig:VarianceandBiasforStochasticrounding}. The probability distribution found is exactly the same as in the SR method. It can be calculated analytically by setting \cref{eq:bias} to 0, so 
\begin{align*}
\left(\lfloor x \rfloor p+(\lfloor x \rfloor+\delta)(1-p)\right)-x=0.
\end{align*}
We find, $p=1-\frac{x-\lfloor x \rfloor}{\delta}$, as in \cref{eq:stroundinggeneral}. From \cref{fig:VarianceandBiasforStochasticrounding}, it can be concluded that the bias is zero for all $x$ and the variance is highest at the tie point. This distribution is repeated for every interval in $\mathcal{X}$.
\begin{figure}[http]
	\centering
	\includegraphics[width=0.6\columnwidth]{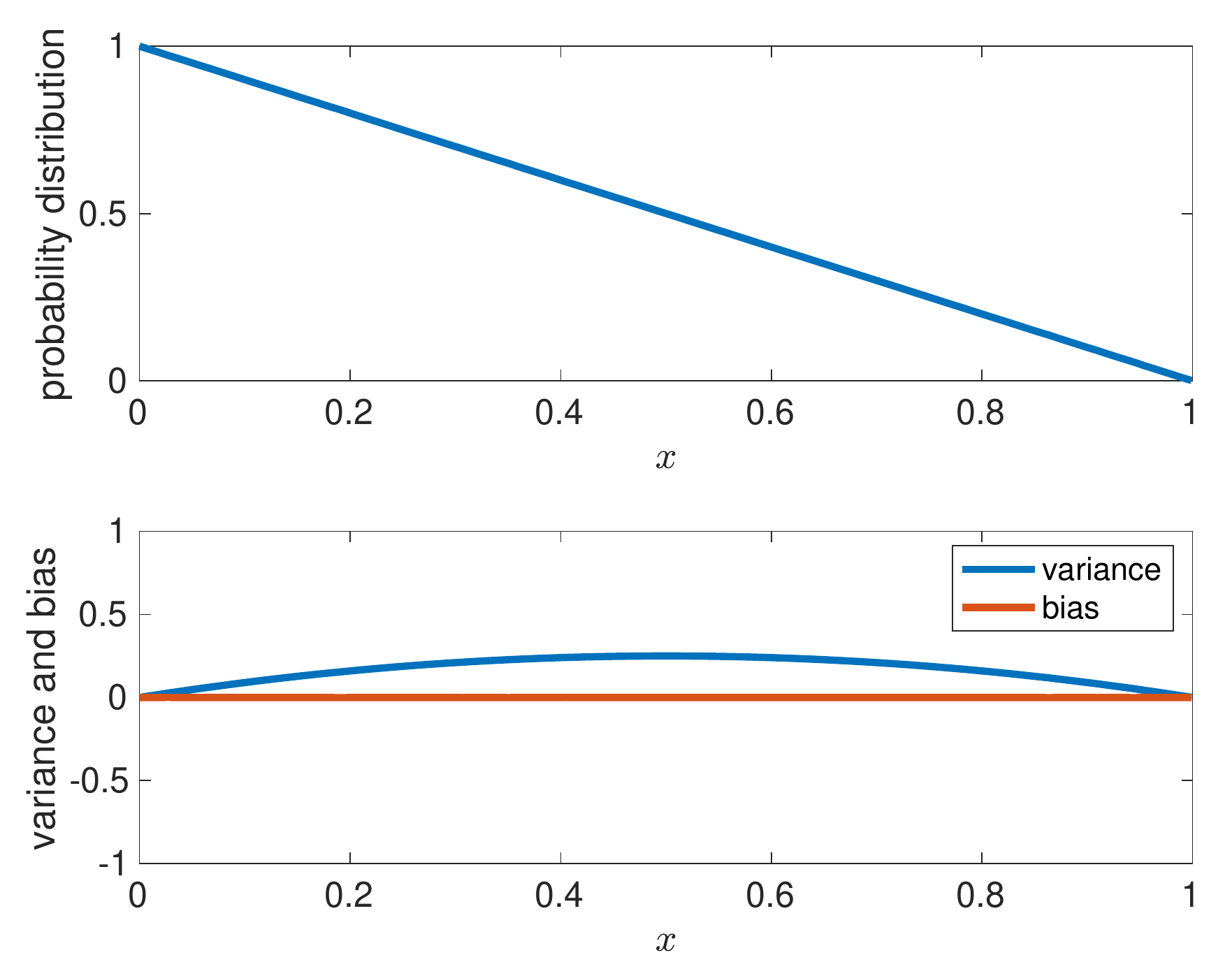}
	\caption{Probability distribution (top) and corresponding variance and bias (bottom) with respect to input variable x, for bias minimization.} 
	\label{fig:VarianceandBiasforStochasticrounding}
\end{figure}
\subsubsection{Variance minimization}
\label{sec:Variance minimization}
\begin{figure}[tbhp]
	\centering
	\subfloat[$p=0$]{\label{fig:a}\includegraphics[width=0.5\textwidth]{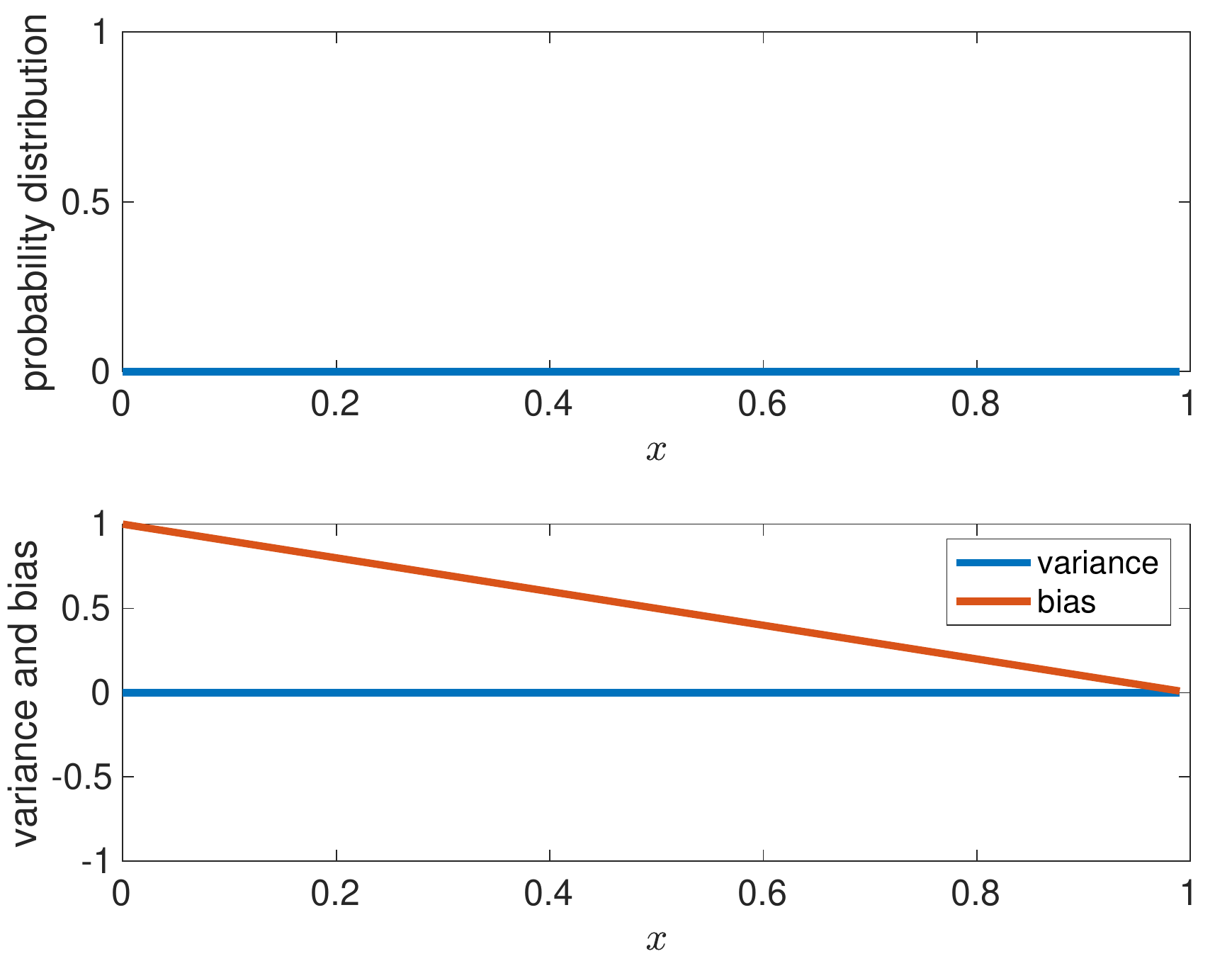}}
	\subfloat[$p=1$]{\label{fig:b}\includegraphics[width=0.5\textwidth]{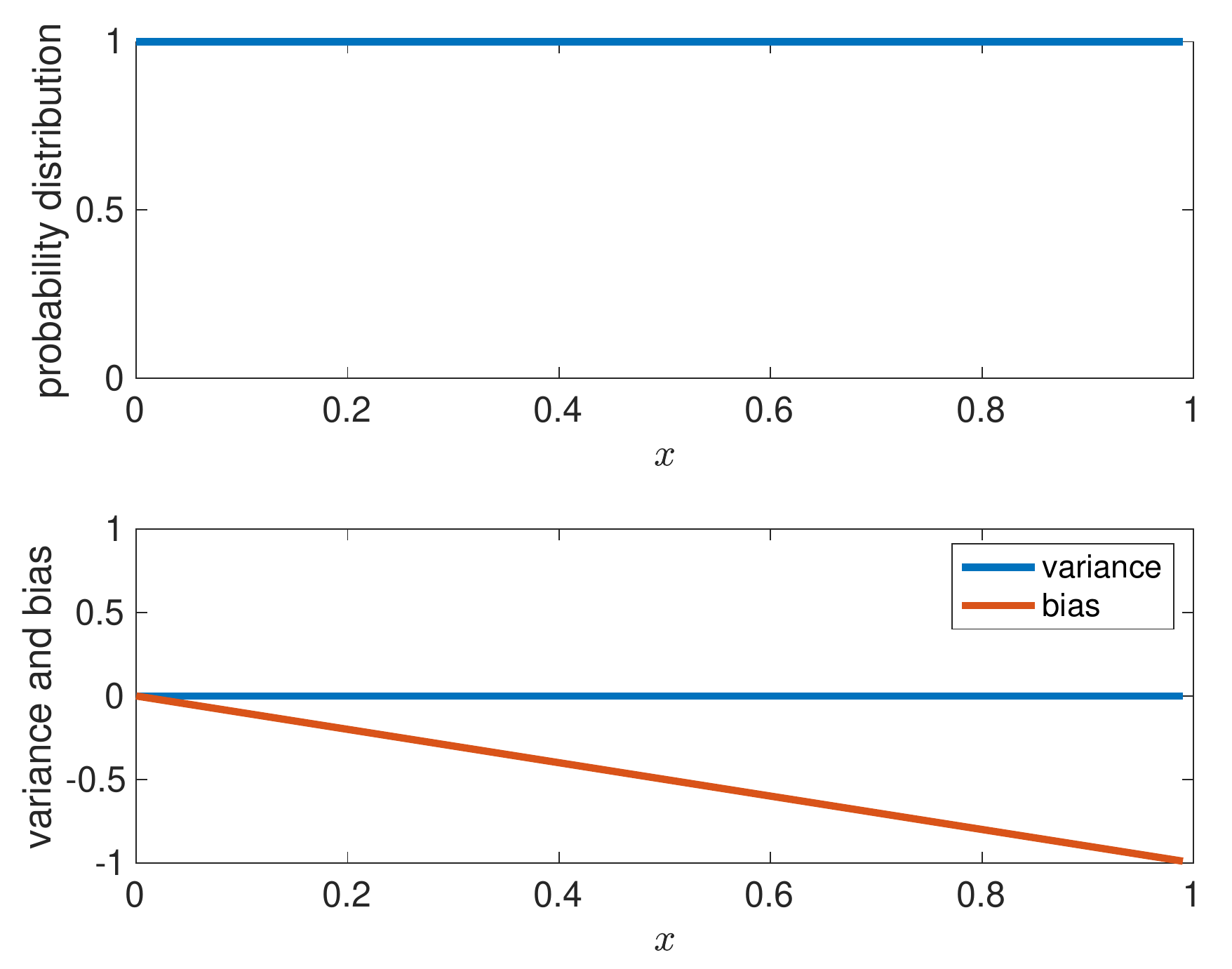}}
	\caption{Probability distribution (top) and corresponding variance and bias (bottom) with respect to input variable x, for variance minimization, when $p=0$ (a) and $p=1$ (b).}
	\label{fig:VarianceandBiasforminimizedvariance}
\end{figure}
\begin{figure}[http]
	\centering
	\includegraphics[width=0.6\columnwidth]{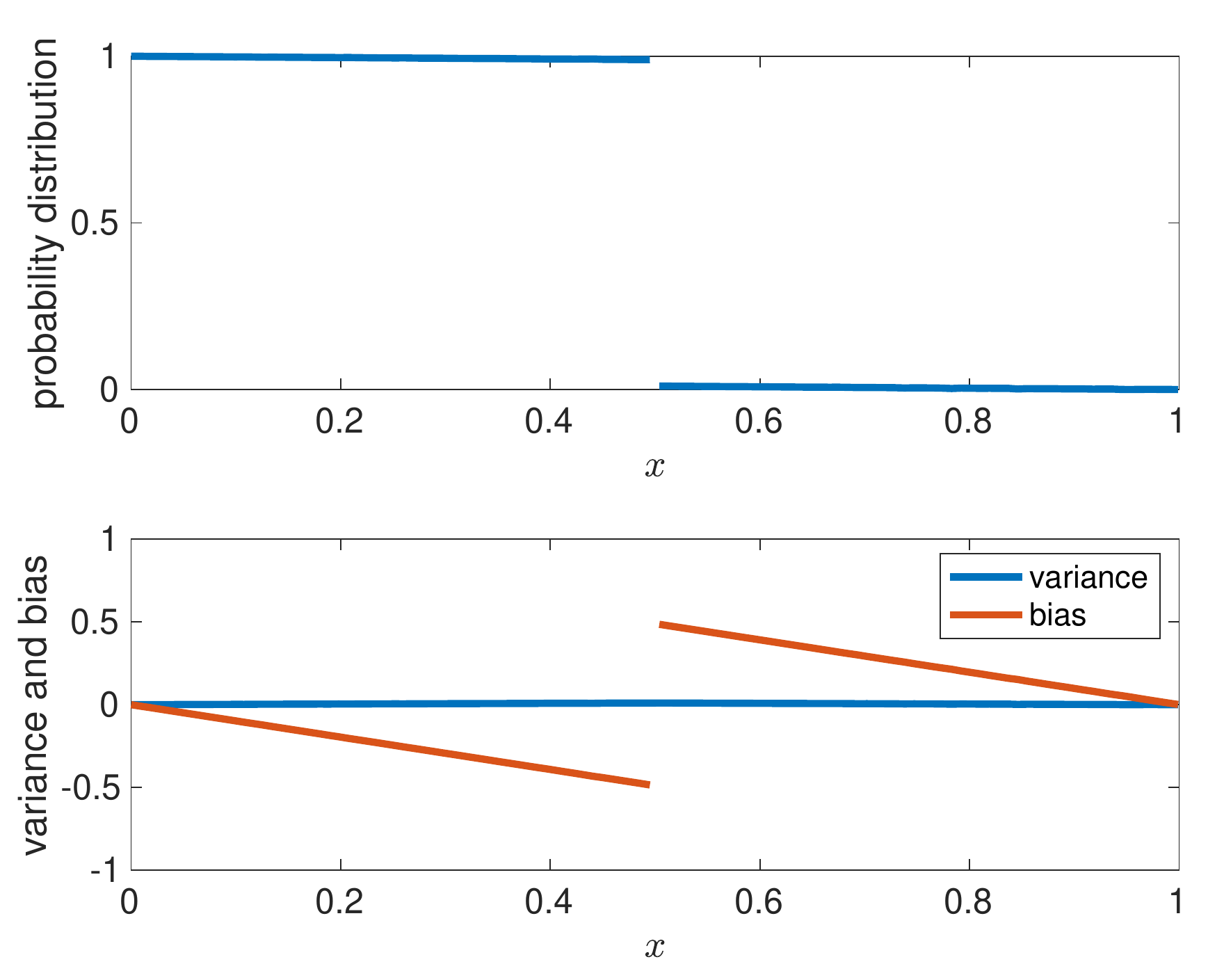}
	\caption{Probability distribution (top) and corresponding variance and bias (bottom) with respect to input variable x, for rounding-to-the-nearest method.} 
	\label{fig:VarianceandBiasforminimizedvariance3}
\end{figure}
When variance is the only objective function, and no constraints are considered, $\theta_2=0$. Based on \cref{eq:variancestr}, the variance is 0 when $p=0$ or $p=1$. These two choices are the ceiling and the floor method, respectively. It can be seen from \cref{fig:VarianceandBiasforminimizedvariance} that the optimal probability distribution of variance minimization is not unique; the probability can be either 0 (\cref{fig:a}) or 1 (\cref{fig:b}), resulting in a bias, which either equals $\delta+\lfloor x \rfloor-x$ or $\lfloor x \rfloor -x$, respectively. If in \cref{eq:optfun}, variance is multiplied with a large parameter and bias with a small parameter, e.g., $\theta_1=0.98$ and $\theta_2=0.02$, we find the probability distribution given in \cref{fig:VarianceandBiasforminimizedvariance3}. The probability distribution is similar to the rounding to the nearest integer with different tie breaking rules such as round half up, round half down, round half to even and round half to odd, as summarized in \cref{tab:roundingmode}. Comparing \cref{fig:VarianceandBiasforminimizedvariance,fig:VarianceandBiasforminimizedvariance3}, the resulting bias using rounding to the nearest is twice as small as that for the floor or ceiling method. Among all the methods of rounding to the nearest integer, CR is the default rounding mode in IEEE 754 floating-point operations and has the probability distribution shown in \cref{fig:VarianceandBiasforminimizedvariance3}. CR will be further studied and compared with stochastic rounding methods in \cref{sec:numericalstudy}.

\subsubsection{Trade-off between bias and variance}
\label{sec:Trade off between bias and variance}
\begin{figure}[tbhp]
	\centering
	\subfloat[]{\label{fig:VarianceandBiasforSameimportance}\includegraphics[width=0.5\textwidth]{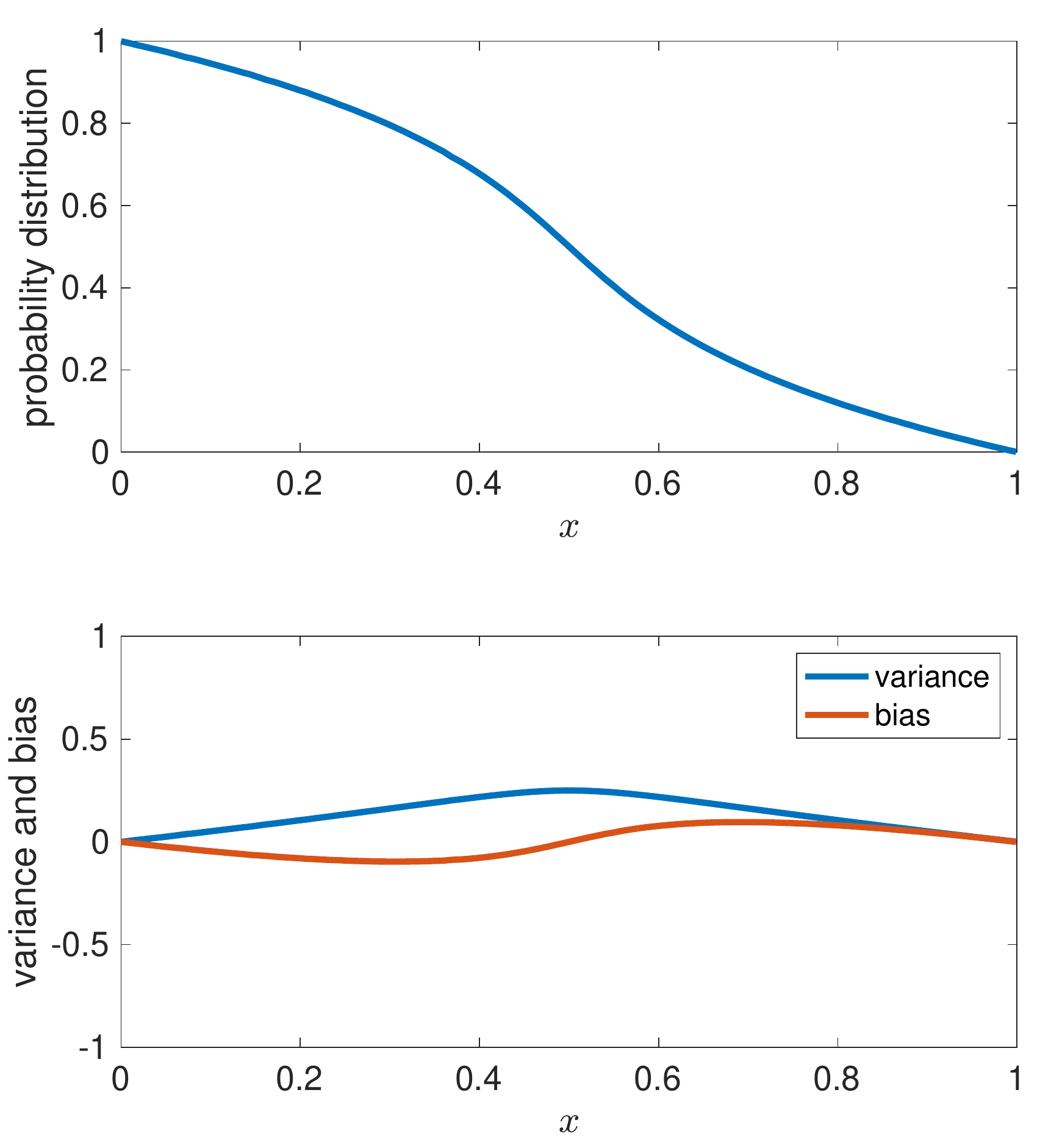}}
	\subfloat[]{\label{fig:VarianceandBiasforconstrainedbias}\includegraphics[width=0.5\textwidth]{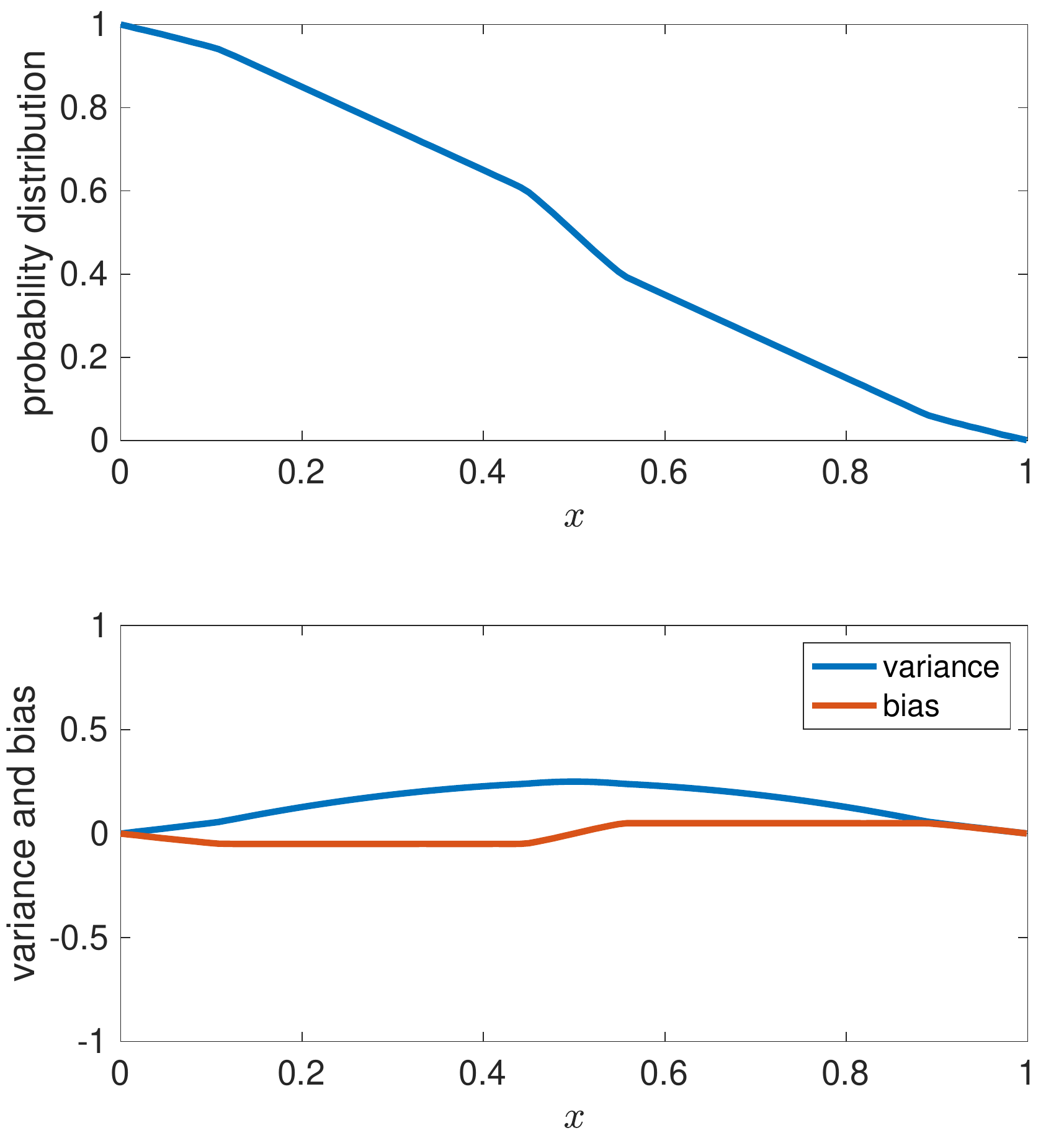}}
	\caption{Probability distribution (top) and corresponding variance and bias (bottom) with respect to input variable x, for both bias and variance minimization (a) and for constrained bias and variance minimization (b).}
\end{figure}
Considering the trade-off between bias and variance, it is challenging to find a good balance between both. Assume bias and variance are equally important in our rounding scheme, which means $\theta_1=\theta_2=0.5$ in \cref{eq:optfun}. This rounding scheme is called Distribution 1, shortly D1, in the remainder of this paper. Under unconstrained conditions, i.e., $k_1=k_2=0$, which means no constraints are set to either bias or variance, the optimized probability distribution and corresponding variance and bias are shown in \cref{fig:VarianceandBiasforSameimportance}. It can be observed that the variance is slightly reduced compared to the SR method but still considerably larger than that of the CR method, vice versa for the bias. Furthermore, the probability distribution is no longer linear.

\subsubsection{Constrained bias}
\label{sec:Constrained bias}
According to \cref{eq:variancestr}, the variance cannot be zero in rounding method \cref{eq:strounding} if $p\neq 0~\text{or}~p\neq1$, but a constraint can be imposed to the bias to guarantee a better rounding result. It should be noted that the maximum value of the bias can be set according to the user requirement. The value of the penalty can be set arbitrarily, as long as it is sufficiently larger than the value of the objective function when $k_i=0$. Assume $B_{\text{max}}=0.05$. To achieve this constraint, a penalty $k_2=10^{10}$ will be added if $g_2\geq 0$.  The resulting rounding scheme is called Distribution 2, shortly D2, in the remainder of this paper. \cref{fig:VarianceandBiasforconstrainedbias} shows the probability distribution and the corresponding variance and bias, where the bias is limited to 0.05 and the variance is slightly larger than for scheme D1. Consequently, the probability distribution is also tailored to meet the requirement. 

The smallest variance is generally achieved with the largest bias and vice versa. Through optimization scheme \cref{eq:optfun}, a trade-off can be easily obtained within constraints.

\section{Numerical experiments}
\label{sec:numericalstudy}
In this section, the rounding methods SR, CR, D1 and D2 are compared with respect to the absolute value of bias ($\vert B \vert$), variance ($V$) and average absolute value of relative error ($e$) in some numerical experiments. First, experiments are done for the sum operation, with different input distributions. Next, the performance for the square root operation is studied, in which Newton's method is employed to iteratively compute the result. Finally, experiments are done for the inner product operation.  
\subsection{Summation}
Due to the tie-breaking rule of CR and the stochastic behavior of SR, D1 and D2 rounding, the distribution of input variables will influence the rounding result. For instance, if the input variables are distributed in $[0,1]$, where we define 0 to be even, the rounding result of $0.5$, using CR, will always be biased, it will always be rounded to the even number 0. To study the influence of different input distributions on the rounding result, in this section, the experiments will be studied with four input distributions:
\begin{itemize}
	\item \textbf{Case I}: repeated numbers distributed in an odd number of intervals;
	\item \textbf{Case II}: repeated numbers distributed in an even number of intervals;
	\item \textbf{Case III}: non-repeated numbers distributed in an odd number of intervals;
	\item \textbf{Case IV}: non-repeated numbers distributed in an even number of intervals.
\end{itemize}
\subsubsection{Generation of input numbers}
If numbers are distributed uniformly, the probability of the presence of repeated numbers depends on the number of samples in each interval, $N_s$. Specifically, a large value of $N_s$ leads to a larger probability of repeated numbers. In our simulation study, the repeated numbers are obtained randomly using a large number of samples in a small interval, and vice versa for non-repeated numbers. For Case I, a set of uniformly distributed random numbers is generated between $[0,1]$, using the Matlab function \code{rand}, where the number of samples equals $N_s=10,000$. The same amount of numbers are randomly generated, for Case II, in $[0,2]$. For Case III, only 10 samples are randomly generated in $[0,1]$, to avoid repeated numbers. Additionally, 20 samples are generated in $[0,2]$ for Case IV. It should be noted that the input numbers are only generated once for each case, and then kept fixed for the different rounding methods.

\subsubsection{Numerical test}
To each of the aforementioned input distributions, we apply the summation operation given by
\begin{equation}
y=\sum_{i=1}^{N_s} x_i.\nonumber
\end{equation}
According to \cref{prop: roundsum}, the rounding result may be reformulated as 
\begin{align}
\mathrm{fl}(y)&=\mathrm{fl}(\mathrm{fl}(\mathrm{fl}(x_1)+x_2)+ \dots + x_{N_s})\nonumber\\
&=\mathrm{fl}(x_1)+\mathrm{fl}(x_2)+ \dots +\mathrm{fl}(x_{N_s}).
\label{eq:roundedsum}
\end{align}
Variance is computed according to \cref{eq:variancegen} and bias is calculated using 
\begin{equation}
B(y)=\mathbb{E}\left(\mathrm{fl}(y)\right)-\mathbb{E}(y).\nonumber
\end{equation}
Additionally, the average absolute value of the relative error is defined as
\begin{equation}
e =\frac{1}{N}\sum_{i=1}^N \frac{\vert \mathrm{fl}(y_i)-y_i \vert }{y_i},\nonumber
\end{equation}
where $N$ is the number of repetitions of the experiment. In this study, 10,000 repetition of experiments are made with the same input for all stochastic rounding methods. All the summation outcomes are rounded to integers for each rounding process. \cref{fig:barbias} shows the normalized absolute value of the bias for rounding methods SR, CR, D1 and D2, for Cases I-IV. The largest bias is always obtained by CR, if the input variables are repeated or distributed in an odd number of intervals, for Cases I-III. This can be explained by the optimization result given in \cref{sec:fourdistribution}, where the bias of CR is larger than SR, D1 and D2. For Case IV, CR realizes an unbiased result, since the input variables are distributed in an even number of intervals, in such a way that the rounding bias in the interval $[0,1]$ is compensated by that in $[1,2]$. A small bias is obtained by SR for Cases I and II, owing to the unbiased property of SR. In general, the biases caused by D2 are always smaller than D1, which is an obvious result led by the constraint on the bias in D2. The variance shows results opposite to those of the bias, as depicted in \cref{fig:barvariance}. The variance of CR is zero for all four cases, since CR is deterministic. The second smallest variance is always obtained by D1, and the largest variance is obtained by SR, though the difference between rounding methods SR, D1, D2 is minor. This agrees with the optimization results in \cref{sec:fourdistribution}. \cref{fig:barre} shows the normalized average absolute value of relative error for the four rounding methods for Cases I-IV. For the repeated input variables, Cases I and II, CR has the largest average absolute value of relative error. For the non-repeated input variables, Cases III and IV, the smallest average absolute value of relative error is achieved by CR and the largest one is obtained by SR.
\begin{figure}[t]
	\centering
	\includegraphics[width=0.6\columnwidth]{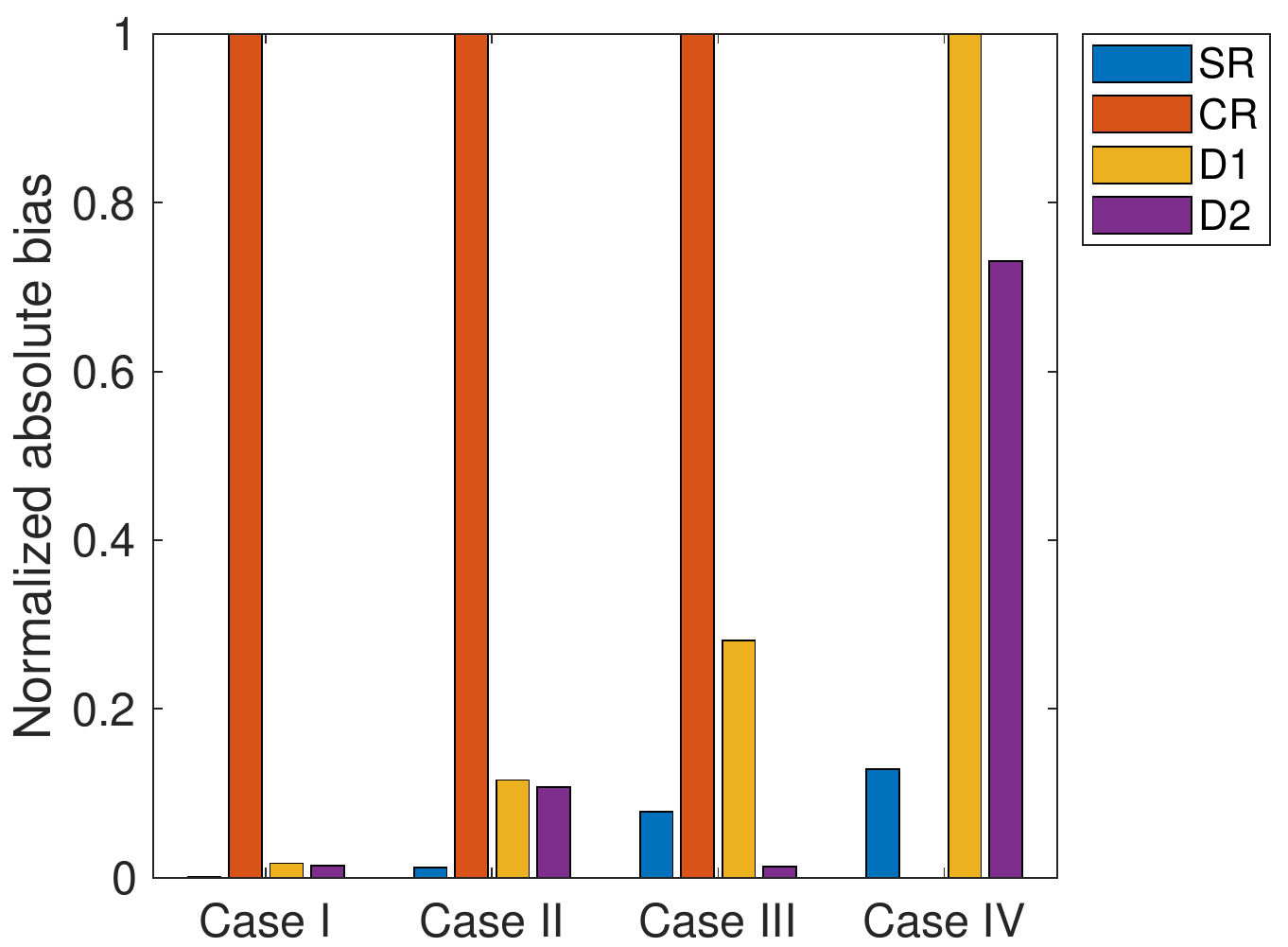}
	\caption{Normalized absolute bias of rounding results using SR, CR, D1 and D2, for Cases I-IV.}\label{fig:barbias}
\end{figure}
\begin{figure}[t]
	\centering
	\subfloat[]{\label{fig:barvariance}\includegraphics[width=0.5\textwidth]{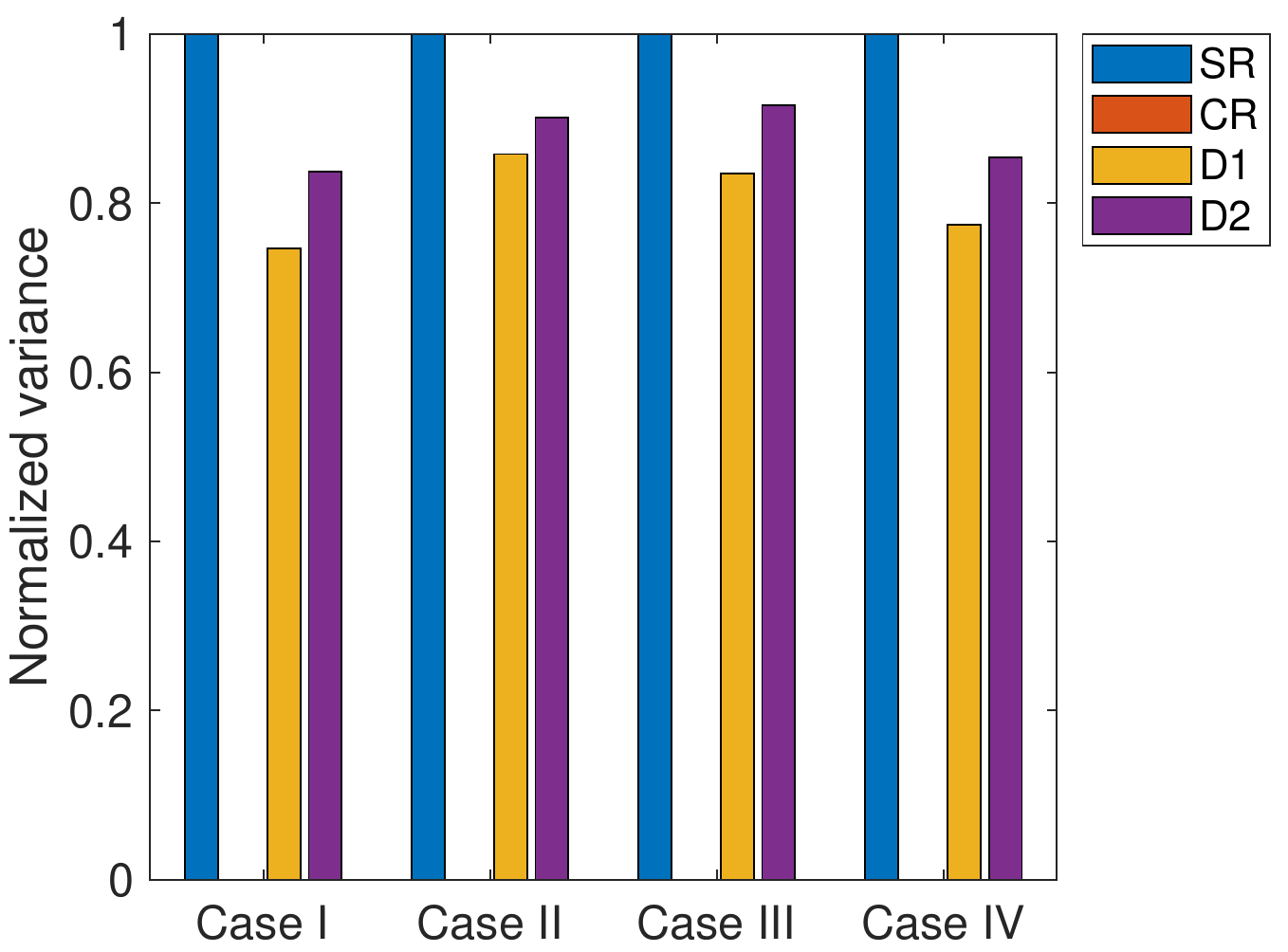}}
	\subfloat[]{\label{fig:barre}\includegraphics[width=0.5\textwidth]{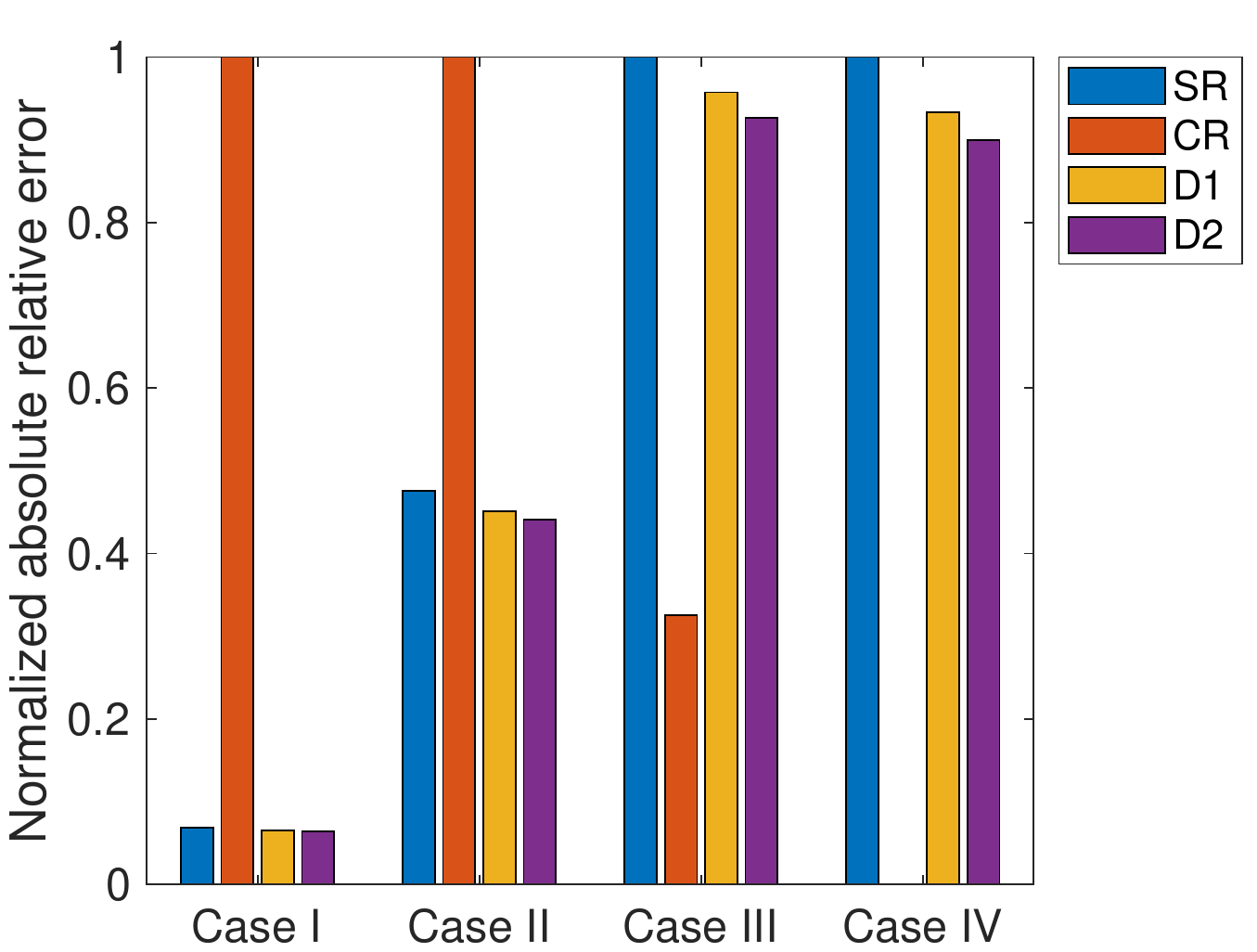}}
	\caption{Normalized variance (a) and absolute relative error (b) of rounding results using SR, CR, D1 and D2, for Cases I-IV.}
\end{figure}
\begin{table}[t]
	{\footnotesize
		\captionsetup{position=top} 
		\footnotesize\caption{Non-normalized values of bias, variance and absolute relative error of SR, CR, D1 and D2, for Cases I-IV, with largest values of $\vert B \vert $, $V$ and $e$ marked in red.}\label{tab:specificvalue}
		\begin{center}
			\begin{tabular}{l|l|l|l|l|l}
				\hline
				\multicolumn{2}{c|}{} & Case I & Case II & Case III & Case IV \\ \hline \rule{0pt}{2.3ex}%
				\multirow{4}{*}{$\vert B \vert $} 
				& SR &  \phantom{}$0.4$   &  \phantom{1}$0.82$ & $0.03$ & $0.034$    \\ 
				& CR &  \textcolor{red}{$0.47\cdot 10^{3}$}              & $\textcolor{red}{69.1}$  & \textcolor{red}{$0.4$}   & $0$        \\ 
				& D1 &  \phantom{}$8.05$    &   \phantom{1}$8$ & $0.11$   & \textcolor{red}{$0.22$}    \\ 
				& D2 &  \phantom{}$6.84$    & \phantom{1}$7.42$   & $0.01$   & $0.16$    \\ \hline \rule{0pt}{2.3ex}%
				\multirow{4}{*}{$V$} 
				& SR & \phantom{}\textcolor{red}{$1.73\cdot 10^{3}$}  & \textcolor{red}{\phantom{1}$1.61\cdot 10^{3}$} & \textcolor{red}{$2.29$ } &\textcolor{red}{$3.24$}     \\ 
				& CR &  \phantom{}$0$                    &  \phantom{1}$0$        &  $0$      & $0$         \\ 
				& D1 & \phantom{}$1.29\cdot 10^{3}$  & \phantom{1}$1.38\cdot 10^{3}$ & $1.91$  & $2.51$    \\ 
				& D2 & \phantom{}$1.45\cdot 10^{3}$  & \phantom{1}$1.45\cdot 10^{3}$ & $2.1$  & $2.77$   \\ \hline \rule{0pt}{2.3ex}%
				\multirow{4}{*}{$e $}   
				& SR & \phantom{}$0.0066$   &  \phantom{1}$0.0032$     & \textcolor{red}{$0.2284$}  & \textcolor{red}{$0.0667$}    \\ 
				& CR & \phantom{}\textcolor{red}{$0.0932$}   &  \phantom{1}\textcolor{red}{$0.0069$}     & $0.0741$  & $0$         \\ 
				& D1 & \phantom{}$0.0061$   &  \phantom{1}$0.0031$     & $0.2201$  & $0.0617$    \\ 
				& D2 & \phantom{}$0.0059$   &  \phantom{1}$0.003$    & $0.2086$  & $0.0588$    \\ \hline
			\end{tabular}
		\end{center}
	}
\end{table}

Overall, for the repeated input variables, D1 results in the rounding results with smallest variance among the stochastic rounding methods, as well as small average absolute value of relative error. For the non-repeated input variables, CR performs best in general, with smallest variance and average absolute value of relative error. The non-normalized values of the bias, variance and average absolute value of relative error in \cref{fig:barbias,fig:barvariance,fig:barre} are given in \cref{tab:specificvalue}, in which the largest bias, variance and average absolute value of relative error are marked in red. It can be observed from \cref{tab:specificvalue} that the value of the bias and variance in Cases I and II are much larger than those of Cases III and IV, because the bias and variance in Cases I and II are accumulated by the sequence of summation, where $N_s=10,000$ in \cref{eq:roundedsum}. 

\subsection{Square root calculation using Newton iteration}
\begin{table}[t]
	{\footnotesize
		\captionsetup{position=top} 
		\footnotesize\caption{Non-normalized values of $\mu$, $\vert B\vert$, $V$, $e$ and $N_{\rm{it}}$ of SR, CR, D1 and D2 for computing $\sqrt{a}$ using Newton's method, with $\delta=10^{-3}$, with largest values of $\vert B\vert$ and $e $ marked in red.}\label{tab:sqrttest}
		\begin{center}
			\begin{tabular}{l|l|l|l|l|l|l}
				\hline
				\multicolumn{2}{c|}{$a$} & $0.30146$ & $6.55501$ & $51.16904$ & $357.00272$ & $8133.27762$\\ \hline \rule{0pt}{2.3ex}%
				\multirow{4}{*}{$\mu$}
				& SR &  \phantom{}$0.5491$  &  \phantom{}$2.5601$ &  \phantom{1}$7.1531$   & \phantom{1}$18.8946$   &  \phantom{12}$90.1849$  \\ 
				& CR &  \phantom{}$0.548$  &  \phantom{}$2.56$ &  \phantom{1}$7.154$   & \phantom{1}$18.894$   &  \phantom{12}$90.184$  \\ 
				& D1 &  \phantom{}$0.549$  &  \phantom{}$2.5604$ &  \phantom{1}$7.153$   & \phantom{1}$18.8942$   &  \phantom{12}$90.1848$  \\ 
				& D2 &  \phantom{}$0.549$  &  \phantom{}$2.56$ &  \phantom{1}$7.1535$   & \phantom{1}$18.894$   &  \phantom{12}$90.1848$  \\
				\hline \rule{0pt}{2.3ex}%
				\multirow{4}{*}{$\vert B \vert $} 
				& SR &  $6.46\cdot 10^{-5}$  &  $1.58\cdot 10^{-4}$  &  \phantom{1}$1.96\cdot 10^{-4}$ & \phantom{12}$1.29\cdot 10^{-4}$   &  \phantom{123}$1.9\cdot 10^{-4}$  \\ 
				& CR &  \textcolor{red}{$1.05\cdot 10^{-3}$}  &  \textcolor{red}{$2.75\cdot 10^{-4}$}  &  \phantom{1}\textcolor{red}{$7.46\cdot 10^{-4}$} & \phantom{12}\textcolor{red}{$5.16\cdot 10^{-4}$}   &  \phantom{123}\textcolor{red}{$6.86\cdot 10^{-4}$}  \\ 
				& D1 &  $5.42\cdot 10^{-5}$  &  $8.55\cdot 10^{-5}$  &  \phantom{1}$2.54\cdot 10^{-4}$ & \phantom{12}$2.8\cdot 10^{-4}$   &  \phantom{123}$1.44\cdot 10^{-4}$  \\ 
				& D2 &  $2.79\cdot 10^{-5}$  & $2.74\cdot 10^{-4}$   &  \phantom{1}$2.41\cdot 10^{-4}$ & \phantom{12}$5.04\cdot 10^{-4}$   &  \phantom{123}$1.91\cdot 10^{-5}$  \\
				\hline \rule{0pt}{2.3ex}%
				\multirow{4}{*}{$V$} 
				& SR &  $8.2\cdot 10^{-8}$  & $1.04\cdot 10^{-7}$   &  \phantom{1}$4.48\cdot 10^{-8}$  & \phantom{12}$2.29\cdot 10^{-7}$   &   \phantom{123}$1.08\cdot 10^{-7}$ \\ 
				& CR &  \phantom{}$0$     &   \phantom{}0  &  \phantom{1}$0$ &  \phantom{12}$0$   &  \phantom{123}$0$ \\ 
				& D1 & $1.1\cdot 10^{-9}$   &  $2.31\cdot 10^{-7}$  &  \phantom{1}$3.33\cdot 10^{-11}$ & \phantom{12}$1.8\cdot 10^{-7}$   &   \phantom{123}$1.41\cdot 10^{-7}$ \\ 
				& D2 & $1.07\cdot 10^{-7}$   & $1.4\cdot 10^{-9}$   &  \phantom{1}$2.5\cdot 10^{-7}$  & \phantom{12}$1.17\cdot 10^{-8}$   &   \phantom{123}$2.08\cdot 10^{-7}$ \\ 
				\hline \rule{0pt}{2.3ex}%
				\multirow{4}{*}{$e $}   
				& SR & $1.18\cdot 10^{-4}$   &   $6.16\cdot 10^{-5}$  & \phantom{1}$2.73\cdot 10^{-5}$   & \phantom{12}$6.85\cdot 10^{-6}$     &  \phantom{123}$2.11\cdot 10^{-7}$  \\ 
				& CR & \textcolor{red}{$1.92\cdot 10^{-3}$}   &  \textcolor{red}{$1.08 \cdot 10^{-4}$}  & \phantom{1}\textcolor{red}{$1.04\cdot 10^{-4}$}   & \phantom{12}\textcolor{red}{$2.73\cdot 10^{-5}$}   &  \phantom{123}\textcolor{red}{$7.61\cdot 10^{-6}$}  \\ 
				& D1 & $9.88\cdot 10^{-5}$   &  $3.34 \cdot 10^{-5}$  & \phantom{1}$3.55\cdot 10^{-5}$   & \phantom{12}$1.48\cdot 10^{-5}$   & \phantom{123}$1.6\cdot 10^{-6}$   \\ 
				& D2 & $5.09\cdot 10^{-5}$   &  $1.07 \cdot 10^{-4}$  & \phantom{1}$3.36\cdot 10^{-5}$   & \phantom{12}$2.67\cdot 10^{-5}$   &  \phantom{123}$2.11\cdot 10^{-7}$  \\ 
				\hline \rule{0pt}{2.3ex}%
				\multirow{4}{*}{$N_{\rm{it}}$}   
				& SR &  \phantom{}$5.57$  & \phantom{}$5.75$  &  \phantom{1}$7.7$  &  \phantom{12}$9$  & \phantom{12}$11.83$   \\ 
				& CR &  \phantom{}$4$ & \phantom{}$5$  &  \phantom{1}$7$  &  \phantom{12}$8$  & \phantom{12}$11$   \\ 
				& D1 &  \phantom{}$5.66$  & \phantom{}$5.8$  &  \phantom{1}$7.44$  &  \phantom{12}$9$  & \phantom{12}$11.9$   \\ 
				& D2 &  \phantom{}$5.58$  & \phantom{}$5.7$  &  \phantom{1}$7.43$  &  \phantom{12}$9$  & \phantom{12}$11.86$   \\ \hline
			\end{tabular}
		\end{center}
	}
\end{table}
Square root calculation is an approximation process on most processor units such as CPUs, GPUs and FPGAs. It is based on different numerical algorithms. The speed of square root computation is crucial in hardware applications. In this section, the square root operation is studied with different rounding methods, using Newton iteration. The computation precision and speed are studied by rounding to a specific number of decimal digits and implementing integer arithmetic. 

The square root of a number, $\sqrt{a}$, can be iteratively computed by Newton's method by introducing the function $f(x)=x^2-a$, and the iterative process is given by $x_{k+1}=x_{k}-\frac{f(x_k)}{f'(x_k)}=\frac{1}{2}\big(x_k+\frac{a}{x_k}\big)$.
The rounding process can be reformulated as
\begin{align}
\mathrm{fl}(x_{k+1})&=\mathrm{fl}\Big(\frac{1}{2}\Big(\mathrm{fl}(x_k)+\mathrm{fl}\big(\frac{\mathrm{fl}(a)}{\mathrm{fl}(x_k)}\big)\Big)\Big).
\label{eq:sqrtfunround}
\end{align}

In the numerical tests, some random numbers with five decimal digits will be generated, one number in successively $(0, 1)$, $(1, 10)$, $(10, 100)$, $(100, 1000)$ and $(1000, 10000)$. Next, $\sqrt{a}$ is calculated $N$ times, using different rounding schemes, to calculate $\mu$, $B$, $V$, $e$ and the average number of iteration steps ($N_{\rm{it}}$). For each rounding process, the numbers are rounded to three decimal digits. Specifically, the rounding precision is $\delta=10^{-3}$. In Newton's method, the initial guess is $x_0=1$ and $10^{-5}$ is set as the tolerance for convergence and $N_{\rm{max}}=100$ is set as the maximum number of iteration steps to take. A summary of the simulation results is given in \cref{tab:sqrttest}. It can be observed that the largest bias (marked in red) is always obtained by CR. Consequently, the largest average absolute value of relative error (marked in red) is always obtained by CR as well. Still, the value of the largest average absolute value of relative error obtained by CR is generally less than $10^{-3}$, indicating a good accuracy. Furthermore, the resulting average number of iteration steps of CR is integer, since it is a deterministic process. Additionally, D1 shows the most reliable performance of the four rounding schemes, where the average absolute value of relative error of the approximated results using D1 is consistently smaller than $10^{-4}$.
\begin{table}[t]
	{\footnotesize
		\captionsetup{position=top} 
		\footnotesize \caption{Non-normalized values of $\mu$, $\vert B\vert$, $V$, $e$ and $N_{\rm{it}}$ of SR, CR, D1 and D2 for computing $\sqrt{a}$ using Newton's method, with $\delta=1$, with smallest values of $\vert B\vert$ and $e$ marked in blue. If the Newton iteration breaks down or does not converge, we write a dash ($-$). }\label{tab:sqrttestinteger}
		\begin{center}
			\begin{tabular}{l|l|c|l|l|l|l}
				\hline
				\multicolumn{2}{l|}{$a$} & $0.30146$ & \phantom{12}$6.55501$ & $51.16904$ & $357.00272$ & $8133.27762$\\ \hline \rule{0pt}{2.3ex}
				\multirow{4}{*}{$\mu$}
				& SR &  $-$    &  \phantom{12}$2.0966$ &  \phantom{1}$7$    & \phantom{1}$19$    &  \phantom{12}$90.7080$  \\ 
				& CR &  $-$     &  \phantom{12}$3$ &  \phantom{1}$7$    & \phantom{1}$19$    &  \phantom{12}$90$  \\ 
				& D1 &  $-$     &  \phantom{12}$2.5913$ &  \phantom{1}$7.0274$    & \phantom{1}$19$   &  \phantom{12}$90.0032$  \\ 
				& D2 &  $-$     &  \phantom{12}$2.9992$ &  \phantom{1}$7$    & \phantom{1}$19$    &  \phantom{12}$90.0004$  \\ \hline \rule{0pt}{2.3ex}
				\multirow{4}{*}{$\vert B\vert $} 
				& SR &  $-$   &  \phantom{12}$0.46$  &  \phantom{1}$0.15$ & \phantom{12}$0.11$   &  \phantom{123}$0.523$  \\ 
				& CR &  $-$   &  \phantom{12}$0.44$  &  \phantom{1}$0.15$ & \phantom{12}$0.11$   &  \phantom{123}$0.185$  \\ 
				& D1 &  $-$   &  \phantom{12}\textcolor{blue}{$0.03$}    &  \phantom{1}\textcolor{blue}{$0.13$} & \phantom{12}$0.11$   &  \phantom{123}\textcolor{blue}{$0.182$}  \\ 
				& D2 &  $-$   &  \phantom{12}$0.44$  &  \phantom{1}$0.15$ & \phantom{12}$0.11$   &  \phantom{123}$0.184$  \\ \hline \rule{0pt}{2.3ex}
				\multirow{4}{*}{$V$} 
				& SR &  $-$   &  \phantom{12}$0.09$   &  \phantom{1}$3.33\cdot 10^{-5}$   & \phantom{12}$0$   &   \phantom{123}$0.21$ \\ 
				& CR &  $-$   &  \phantom{12}$0$        &  \phantom{1}$0$                       & \phantom{12}$0$   &   \phantom{123}$0$ \\ 
				& D1 &  $-$    &  \phantom{12}$0.24$  &  \phantom{1}$0.03$                  & \phantom{12}$0$   &   \phantom{123}$3.2\cdot 10^{-3}$ \\ 
				& D2 &  $-$    &  \phantom{12}$8\cdot 10^{-4}$  &  \phantom{1}$0$                       & \phantom{12}$0$   &   \phantom{123}$4\cdot 10^{-4}$ \\ \hline \rule{0pt}{2.3ex}
				\multirow{4}{*}{$e $}   
				& SR & $-$    &  \phantom{12}$0.18$  & \phantom{1}0.021   & \phantom{12}$5.58\cdot 10^{-3}$     &  \phantom{123}$5.8\cdot 10^{-3}$  \\ 
				& CR & $-$    &  \phantom{12}$0.17$  & \phantom{1}$0.021$   & \phantom{12}$5.58\cdot 10^{-3}$     &  \phantom{123}$2.05\cdot 10^{-3}$  \\ 
				& D1 & $-$    &  \phantom{12}\textcolor{blue}{$0.02$}  & \phantom{1}\textcolor{blue}{$0.018$}   & \phantom{12}$5.58\cdot 10^{-3}$     &  \phantom{123}\textcolor{blue}{$2.01\cdot 10^{-3}$}   \\ 
				& D2 & $-$    &  \phantom{12}$0.17$  & \phantom{1}$0.021$   & \phantom{12}$5.58\cdot 10^{-3}$     &  \phantom{123}$2.04\cdot 10^{-3}$  \\ \hline \rule{0pt}{2.3ex}
				\multirow{4}{*}{$N_{\rm{it}}$}   
				& SR & $-$   &  \phantom{12}$3.71$    &  \phantom{1}$5.49$  &  \phantom{12}$7.27$  & \phantom{123}$9.99$  \\ 
				& CR & $-$   &  \phantom{12}$-$              &  \phantom{1}$6$ &  \phantom{12}$7$  & \phantom{12}$10$  \\ 
				& D1 & $-$   &  \phantom{12}$3.81$    &  \phantom{1}$5.45$  &  \phantom{12}$7.16$  & \phantom{123}$9.97$  \\ 
				& D2 & $-$   &  \phantom{12}$3.79$    &  \phantom{1}$5.47$  &  \phantom{12}$7.21$  & \phantom{123}$9.98$  \\ \hline
			\end{tabular}
		\end{center}
	}
\end{table}

If integer arithmetic is considered in \cref{eq:sqrtfunround}, i.e., when rounding precision $\delta=1$ is applied, repeating the same calculations as in \cref{tab:sqrttest}, the results given in \cref{tab:sqrttestinteger} are obtained. It can be observed that, as expected, the square root of numbers between $[0,1]$ is not solvable using integer arithmetic, since the rounding result of $x_k$ and $a$ can be 0. Moreover, Newton's method fails to converge using the CR method, for 6.55501, using integer arithmetic. It should be noted that the aforementioned results are specific for the numbers given in \cref{tab:sqrttest}, since the numbers are randomly generated. However, in general, rounding method D1 offers the smallest bias and average absolute value of relative error (both marked in blue). Stochastic rounding methods, such as SR, D1 and D2, guarantee faster convergence than CR. Comparing \cref{tab:sqrttest} and \cref{tab:sqrttestinteger}, the approximated square root using rounding to three decimal digits ($\delta=10^{-3}$) is more accurate than that using integer arithmetic, while a faster convergence is obtained by integer arithmetic for all stochastic rounding methods. 

\subsection{Inner product computation}
\label{sec:innerproduct}
In this section, some tests will be performed using different rounding methods for inner product computation. For two vectors ${\bf x}$ and $\mathbf{y}$, the inner product of ${\bf x}$ and $\mathbf{y}$ can be calculated as $\langle\mathbf{x},~\mathbf{y}\rangle=x_1y_1+x_2y_2+\dots +x_{N_s}y_{N_s}$.
Considering the rounding process, it can be formulated as
\begin{equation}
\mathrm{fl}(\langle\mathbf{x},\mathbf{y}\rangle)=\mathrm{fl}\big(\mathrm{fl}(x_1)\mathrm{fl}(y_1)\big)+~\mathrm{fl}\big(\mathrm{fl}(x_2)\mathrm{fl}(y_2)\big)+\dots +~\mathrm{fl}\big(\mathrm{fl}(x_{N_s})\mathrm{fl}(y_{N_s})\big).
\label{eq:innerproduct_round}
\end{equation}
Based on \cref{prop:multiplication}, when rounding to integers is applied, \cref{eq:innerproduct_round} can be simplified to
\begin{equation}
\mathrm{fl}(\langle\mathbf{x},\mathbf{y}\rangle)=\mathrm{fl}(x_1)\mathrm{fl}(y_1)+~\mathrm{fl}(x_2)\mathrm{fl}(y_2)+\dots +~\mathrm{fl}(x_{N_s})\mathrm{fl}(y_{N_s}),
\label{eq:innerproduct_roundtointeger}
\end{equation}
thus the number of roundings is reduced $N_s$ times.

According to \cref{prop:worsterrorbound_multi,remk:remark1}, the presence of small numbers will lead to relative round-off errors larger than 1, when implementing the operation of multiplication. If the inner product is close to zero, the relative errors of inner product can be very large. To enable a better comparison of performance between different rounding methods, in the simulation study, the input vectors are designed to have a low composition of small numbers and the inner product is supposed to be not close to zero. Hence, the vector $\mathbf{x}$ is generated using the sine function with input vector $\mathbf{y}$, where $\mathbf{y}$ is comprised of $N_s$ points distributed equidistantly in $[0,2\pi]$. The sine function is a more appropriate choice than for instance the cosine function, since the combination of $\cos(\mathbf{y})$ and $\mathbf{y}$ has a larger chance to result in the worst-case relative round-off error when using SR. E.g., when $y\to \frac{\pi}{2}$ and $y\to \frac{3\pi}{2}$, where $y\in [0,2\pi]$, we have $x=\cos(y)\to 0$. According to \cref{remk:remark1,fig:contourplot}, the probability of getting a large worst-case relative round-off error will be $\frac{\pi}{2}-\lfloor \frac{\pi}{2} \rfloor \approx 0.57$ and $\frac{3\pi}{2}-\lfloor\frac{3\pi}{2}\rfloor \approx 0.71$, while the probabilities for the sine function are $0.14$ and $0.28$. Furthermore, $\int_{0}^{2\pi}y\cos(y)dy=0$, which also indicates that the inner product is close to 0.
 \begin{table}[t]
 	{\footnotesize
 		\captionsetup{position=top} 
 		\footnotesize \caption{Non-normalized values of $\vert B\vert$, $V$ and $e$ of SR, CR, D1 and D2 for computing inner products using integer arithmetic, with largest values of $\vert B\vert$, $V$ and $e$ and smallest values of $\vert B\vert$ and $e$ marked in red and blue, respectively. }\label{tab:innerproduct}
 		\begin{center}
 			\begin{tabular}{c|c|l|l|l|l|l|l}
 				\hline    
 				\multicolumn{2}{c|}{$N_s$} & \multicolumn{1}{c|}{$50$} & \multicolumn{1}{c|}{$200$} & \multicolumn{1}{c|}{$400$} &  \multicolumn{1}{c|}{$600$} & \multicolumn{1}{c|}{$800$} &  \multicolumn{1}{c}{$1000$}\\ \hline \rule{0pt}{2.3ex}
 				\multirow{4}{*}{$\vert B \vert $} 
 				& SR &  \phantom{1}$0.17$ & \phantom{12}\textcolor{blue}{ $0.17$} &  \phantom{12}\textcolor{blue}{$0.11$}    &  \phantom{1}\textcolor{blue}{$0.27$} & \phantom{1}\textcolor{blue}{ $0.44$} & \phantom{1}\textcolor{blue}{$0.75$}\\ 
 				& CR &  \phantom{1}\textcolor{blue}{$0.07$} &  \phantom{12}$9.02$ &  \phantom{1}\textcolor{red}{$17.01$}   &  \textcolor{red}{$29.01$}&  \textcolor{red}{$35$} &\textcolor{red}{$44$}\\ 
 				& D1 &  \phantom{1}\textcolor{red}{$7.12$} &  \phantom{1}\textcolor{red}{$11.13$} &  \phantom{1}$15.01$  &  $19.3$ & $24.14$ &$29.05$\\ 
 				& D2 &  \phantom{1}$6.93$ &  \phantom{12}$9.9$ &  \phantom{1}$13.53$   &  $16.91$ & $20.09$ &$24.5$\\ \hline \rule{0pt}{2.3ex}
 				\multirow{4}{*}{$V$} 
 				& SR &  \textcolor{red}{$96.02$}   &  \textcolor{red}{$382.6$}  & \textcolor{red}{$768.52$}   &  \phantom{1}\textcolor{red}{$1.14\cdot 10^{3}$} & \phantom{1}\textcolor{red}{$1.49\cdot 10^{3}$} &\phantom{1}\textcolor{red}{$1.94\cdot 10^{3}$}\\ 
 				& CR &  \phantom{1}$0$        &  \phantom{12}$0$         &  \phantom{12}$0$          &  \phantom{1}$0$    &\phantom{1}$0$        & \phantom{1}$0$\\ 
 				& D1 &  $75.02$   &  $305.26$  & $609.2$   &   \phantom{1}$0.91\cdot 10^{3}$            &\phantom{1}$1.22\cdot 10^{3}$  &\phantom{1}$1.52\cdot 10^{3}$\\ 
 				& D2 &  $80.61$   &  $327.37$  & $654.95$   &   \phantom{1}$0.99\cdot 10^{3}$            &\phantom{1}$1.29\cdot 10^{3}$  &\phantom{1}$1.65\cdot 10^{3}$\\ \hline \rule{0pt}{2.3ex}
 				\multirow{4}{*}{$e $}   
 				& SR & \phantom{1}$0.161$   & \phantom{12}$0.078$   &  \phantom{12}$0.055$     &  \phantom{1}\textcolor{blue}{$0.045$}  &\phantom{1}\textcolor{blue}{$0.039$}  &\phantom{1}\textcolor{blue}{$0.035$}\\ 
 				& CR & \phantom{1}\textcolor{blue}{$0.001$}   & \phantom{12}\textcolor{blue}{$0.045$}   &  \phantom{12}\textcolor{blue}{$0.043$}     &  \phantom{1}\textcolor{red}{$0.048$}  &\phantom{1}\textcolor{red}{$0.044$}  &\phantom{1}\textcolor{red}{$0.044$}\\ 
 				& D1 & \phantom{1}$0.188$   & \phantom{12}\textcolor{red}{$0.084$}   &  \phantom{12}$0.058$    &  \phantom{1}$0.048$  &\phantom{1}$0.043$  &\phantom{1}$0.04$\\ 
 				& D2 & \phantom{1}\textcolor{red}{$0.189$}   & \phantom{12}$0.083$   &  \phantom{12}\textcolor{red}{$0.058$}     &  \phantom{1}$0.047$  &\phantom{1}$0.042$  &\phantom{1}$0.038$ \\ \hline
 			\end{tabular}
 		\end{center}
 	}
 \end{table}

 \cref{tab:innerproduct} shows the bias, variance and average absolute value of relative error of \cref{eq:innerproduct_roundtointeger} for computing inner products using integer arithmetic with the different rounding methods, over 10,000 times repetition, for different vector sizes $N_s$. It can be observed from \cref{tab:innerproduct} that SR guarantees the smallest bias (marked in blue), for $N_s\geq200$, but that it also yields the largest variance (marked in red). The average absolute value of relative error obtained by CR first increases with increasing size of the vector $N_s$, but becomes more or less constant for $N_s>200$, whereas the average absolute values of relative errors obtained by the stochastic rounding methods decrease. For $N_s \leq 400$, CR provides a more accurate rounding result than the stochastic rounding methods. However, for $N_s>400$, the average absolute values of the relative errors caused by stochastic rounding methods decrease and eventually become smaller than the relative error of CR. Overall, if the vector size is large, stochastic rounding methods are the better choice to compute inner products with integer arithmetic, though the variance may be very large. For vectors with small size, CR can provide a rounding result as accurate as SR, where CR has zero variance moreover. Among these stochastic rounding methods, SR always guarantees the smallest bias and average absolute value of relative error. However, when $N_s$ increases, the difference of the average absolute value of relative error between each stochastic rounding method is smaller, while variances obtained by D1 and D2 are both smaller than that obtained by SR. Hence, D1 and D2 are the better options, when variance and bias are both important to the rounding result for calculating inner products with large vector size. It should be noted that the behavior of each rounding method in case of inner product operation is very similar to the behavior for summation. SR guarantees the smallest bias, but along with largest variance. D1 and D2 provide smaller variances but slightly larger bias than SR. To achieve the best rounding results, some prior knowledge of the data set is necessary for choosing a rounding mode. 

\section{Conclusion}
\label{sec:conclusion}
Rounding is an essential step in many computations; round-off errors are unavoidable. Deterministic rounding methods generally suffer from rounding bias, while stochastic rounding methods normally have a large rounding variance. In this paper, a systematic way has been proposed to generate a stochastic rounding method with probability distribution that can provide customized rounding bias and variance. As opposed to the conventional stochastic rounding method, the proposed method enables users to tune the rounding probability in different applications, without introducing much extra computational cost. The probability distribution wished for is obtained by formulating a multi-objective optimization problem, which is solved offline using particle swarm optimization. 

Numerical experiments have been performed to analyze the bias, variance and relative error of different rounding methods, such as conventional stochastic rounding, convergent rounding and stochastic rounding with new probability distributions, by implementing three operations: summation, square root calculation through Newton iteration and inner product computation. It has been shown that the rounding results vary in different operations. The proposed stochastic rounding provides the smallest average absolute value of relative rounding error in summation, for repeated input variables. For non-repeated input variables, the smallest average absolute value of relative error is achieved by convergent rounding. Furthermore, the proposed probability distribution of stochastic rounding also offers the best rounding performance in the square root calculation using Newton's method, where the average absolute value of relative error is consistently smaller than $10^{-4}$ for all test cases. Additionally, stochastic rounding methods lead to a faster convergence than convergent rounding when integer arithmetic is applied, in Newton's method. Moreover, by employing integer arithmetic in inner product operations, the number of rounding processes is largely reduced. Stochastic rounding methods show better rounding results with large vector sizes, whereas convergent rounding provides better approximations with small vector sizes.

\section*{Acknowledgments}
This research was funded by the EU ECSEL Joint Undertaking under grant agreement no.~826452 (project Arrowhead Tools).


\begin{thebibliography}{10}

\bibitem{allton1989stochastic}
{\sc C.~Allton, C.~Yung, and C.~Hamer}, {\em Stochastic truncation method for
  {Hamiltonian} lattice field theory}, Phys. Rev. D, 39 (1989), p.~3772.

\bibitem{caramia2008multi}
{\sc M.~Caramia and P.~Dell'Olmo}, {\em Multi-objective management in freight
  logistics: Increasing capacity, service level and safety with optimization
  algorithms}, Springer Science \& Business Media, 2008.

\bibitem{censor1977pareto}
{\sc Y.~Censor}, {\em Pareto optimality in multiobjective problems}, Appl.
  Math. Opt., 4 (1977), pp.~41--59.

\bibitem{cowlishaw2003decimal}
{\sc M.~F. Cowlishaw}, {\em Decimal floating-point: Algorism for computers}, in
  Proceedings of the 16th IEEE Symposium on Computer Arithmetic, IEEE, 2003,
  pp.~104--111.

\bibitem{gupta2015deep}
{\sc S.~Gupta, A.~Agrawal, K.~Gopalakrishnan, and P.~Narayanan}, {\em Deep
  learning with limited numerical precision}, in Proceedings of the 32nd
  International Conference on Machine Learning, 2015, pp.~1737--1746.

\bibitem{higham2019simulating}
{\sc N.~J. Higham and S.~Pranesh}, {\em Simulating low precision floating-point
  arithmetic.}, Manchester Institute for Mathematical Sciences, 2019,
  \url{http://eprints.maths.manchester.ac.uk/2692/}.

\bibitem{higham2019squeezing}
{\sc N.~J. Higham, S.~Pranesh, and M.~Zounon}, {\em Squeezing a matrix into
  half precision, with an application to solving linear systems}, SIAM J. Sci.
  Comput., 41 (2019), pp.~A2536--A2551.

\bibitem{hopkins2019stochastic}
{\sc M.~Hopkins, M.~Mikaitis, D.~R. Lester, and S.~Furber}, {\em Stochastic
  rounding and reduced-precision fixed-point arithmetic for solving neural
  {ODEs}}, arXiv preprint arXiv:1904.11263,  (2019).

\bibitem{kahan1996ieee}
{\sc W.~Kahan}, {\em {IEEE} standard 754 for binary floating-point arithmetic},
  Lecture Notes on the Status of {IEEE}, 754 (1996),
  \url{http://http.cs.berkeley.edu/~wkahan/ieee754status/ieee754.ps}.

\bibitem{PSO2015}
{\sc F.~Marini and B.~Walczak}, {\em Particle swarm optimization ({PSO}). {A}
  tutorial}, Chemometr. Intell. Lab. Syst., 149 (2015), pp.~153--165.

\bibitem{na2017chip}
{\sc T.~Na, J.~H. Ko, J.~Kung, and S.~Mukhopadhyay}, {\em On-chip training of
  recurrent neural networks with limited numerical precision}, in Proceedings
  of the 2017 International Joint Conference on Neural Networks (IJCNN), IEEE,
  2017, pp.~3716--3723.

\bibitem{nightingale1986gap}
{\sc M.~Nightingale and H.~Bl{\"o}te}, {\em Gap of the linear spin-1
  {Heisenberg} antiferromagnet: A {Monte} {Carlo} calculation}, Phys. Rev. D,
  33 (1986), p.~659.

\bibitem{ortiz2018low}
{\sc M.~Ortiz, A.~Cristal, E.~Ayguad{\'e}, and M.~Casas}, {\em Low-precision
  floating-point schemes for neural network training}, arXiv preprint
  arXiv:1804.05267,  (2018).

\bibitem{price1993stochastic}
{\sc P.~Price, C.~Hamer, and D.~O'Shaughnessy}, {\em Stochastic truncation for
  the (2+1){D} {Ising} model}, J. Phys. A, 26 (1993), p.~2855.

\bibitem{santoro1989rounding}
{\sc M.~R. Santoro, G.~Bewick, and M.~A. Horowitz}, {\em Rounding algorithms
  for {IEEE} multipliers}, in Proceedings of 9th Symposium on Computer
  Arithmetic, IEEE, 1989, pp.~176--183.

\bibitem{wang2018training}
{\sc N.~Wang, J.~Choi, D.~Brand, C.-Y. Chen, and K.~Gopalakrishnan}, {\em
  Training deep neural networks with 8-bit floating point numbers}, in
  Proceedings of the 31st Neural Information Processing Systems Conference,
  2018, pp.~7675--7684.

\bibitem{xia2019constrained}
{\sc L.~Xia, R.~Willems, B.~de~Jager, and F.~Willems}, {\em Constrained
  optimization of fuel efficiency for {RCCI} engines}, IFAC-PapersOnLine, 52
  (2019), pp.~648--653.

\end{thebibliography}
\end{document}